\documentclass[12pt,a4paper]{article}

\usepackage{amsthm}
\usepackage{amssymb}\usepackage{amsmath}
\usepackage{latexsym}
\usepackage{srcltx}
\usepackage[dvips]{graphicx}
\usepackage{epsfig}
\usepackage{url}

\newcommand{\re}{{\mathbb R}}
\newcommand{\n}{{\mathbb N}}

\newcommand{\cA}{{\cal{A}}}
\newcommand{\E}{{\mathbf{E}\, }}

\newcommand{\cK}{{\cal{K}}}

\newcommand{\cB}{{\cal{B}}}

\newcommand{\cM}{{\cal{M}}}

\newtheorem{theorem}{Theorem}
\newtheorem{proposition}{Proposition}
\newtheorem{lemma}{Lemma}
\newtheorem{corollary}{Corollary}
\newtheorem{definition}{Definition}

\newtheorem{remark}{Remark}

\title{Convex Optimization methods for computing  \\ the Lyapunov
Exponent of matrices
\thanks{
The first author is supported by the RFBR grants No 10-01-00293 and No
11-01-00329, and by the grant of Dynasty foundation. This research was carried out while the first author was visiting
Center of Operation Research and Econometrics (CORE) and Universit\'e Catholique de Louvain (Louvain-la-Neuve, Belgium)
in Spring, 2011. That author is grateful to the institute and to the university for their hospitality.
The second author is an F.R.S.-FNRS fellow.}}

\author{V.\,Yu.~Protasov
\thanks{Dept. of Mechanics and Mathematics, Moscow State University,
Vorobyovy Gory, 119992, Moscow, {e-mail: 
v-protassov@yandex.ru.}}
\and R.\,M.~Jungers
\thanks{Universit\'e catholique de Louvain (UCLouvain), ICTEAM institute, 4 avenue Georges Lemaitre, B-1348 Louvain-la-Neuve, Belgium.
e-mail: raphael.jungers@uclouvain.be.}}

\date{}
\begin{document}
\maketitle

\begin{abstract}
We introduce a new approach to evaluate the largest Lyapunov exponent of a family of nonnegative matrices.
The method is based on using special positive homogeneous functionals on $\re^d_+$, which gives iterative lower and upper
bounds for the Lyapunov exponent. They improve previously known bounds and converge
to the real value. The rate of converges is estimated and the efficiency of the algorithm is
demonstrated on several problems from applications (in functional analysis, combinatorics, and language theory) and
on numerical examples with randomly generated matrices. The method computes the Lyapunov exponent with a prescribed accuracy
in relatively high dimensions (up to 60). We generalize this approach to all matrices, not necessarily nonnegative,
derive a new universal upper bound for the Lyapunov exponent, and show that such a lower bound, in general,  does not exist.
\end{abstract}

\section{Introduction}

Let us consider a family $\cA = \{A_1, \ldots , A_m\}$ of linear
operators acting in $\re^d$. To each operator~$A_j$ we associate a positive number~$p_j$
so that $\sum_{j=1}^m p_j= 1$. In the sequel we assume that every family of operators  is
equipped with a family of numbers. Consider a random product  $X_k = A_{d_k}\cdots A_{d_1}$,
where all indices $\{d_j\}$ are independent and identically distributed  random variables; each $d_j$
takes values $1, \ldots , m$ with probabilities $p_1, \ldots , p_m$ respectively. According to
the Furstenberg-Kesten theorem~\cite{FK} the value $\|X_k\|^{1/k}$ converges with probability~$1$
to a number $\rho$, which depends only on the family~$\cA$, i.e., on the operators $\{A_j\}_{j=1}^{m}$  and on the probabilities
$\{p_j\}_{j=1}^m$. The number $\lambda = \log \rho$ is called the {\em largest Lyapunov exponent} of the family~$\cA$.
In this paper we do not deal with other Lyapunov exponents, and, for the sake of simplicity, we  omit the word ``largest''.  This number can be defined by the following
 limit formula
 \begin{equation}\label{LE}
 \lambda \quad = \quad \lim_{k \to \infty} \ \frac1k \ \E \log \ \bigl\| \, A_{d_k}\cdots A_{d_1} \bigr\|\, ,
 \end{equation}
where $\E$ denotes the mathematical expectation. The results of this paper can be extended to more general matrix distributions, but we restrict ourselves to i.i.d. matrices taking values on a finite set $\cA$.

We introduce a new approach for computing the Lyapunov exponent based on using special positive homogeneous functionals
on $\re^d$. The idea is the following: for any such a functional $f$ the minimal and maximal expected value of
$\frac1k\, \log \frac{f(Bx)}{f(x)}\, , \, B \in \cA^k$ over all $x\in \re^d, x \ne 0$, give a lower and upper bound
respectively for $\lambda$. For families of nonnegative matrices those bounds  can be effectively computed and then
optimized over certain families of functionals $f$, which leads to optimal bounds $\beta_k \le \lambda \le \alpha_k$ that, under some general assumptions, rapidly
 converge to $\lambda$ as $k \to \infty$. For every~$k$ both $\alpha_k$ and $\beta_k$ are found by solving unconstrained convex minimization problems. The rate of convergence is proved to be at least linear in~$k$, but  in most of practical examples it is much faster. For dimensions $d$ up to $50$ it usually takes less than $k = 12$ iterations to estimate $\lambda$ with the
 relative precision $1\%$.  All computations take a few minutes on a standard desktop computer.

 In Section~II we describe the new approach for operators with a common invariant cone, then in Section~III we  consider two special families of functionals and the corresponding bounds $\alpha_k$ and~$\beta_k$. In Section~IV it is shown that
 for nonnegative matrices both those bounds can be found and optimized over the corresponding families
 as solutions of certain convex minimization problems. In Theorems~\ref{th10} and~\ref{th20} we prove that under
 some general assumptions on matrices we have ${\alpha_k - \beta_k \, \le \, \frac{C}{k}}$, where
 $C$ is an effective constant. In Section~V this technique
 is extended to all matrices, without the nonnegativity condition. We derive an upper bound
 for $\lambda$, which is sharper than the classical upper bound $\, \frac1k\, \E \, \{\log \|B\|_2 \ | \ B \in \cA^k\}$.
 On the other hand, Theorem~\ref{th30} proved in that section shows that there are no good lower bounds for the Lyapunov exponent of general matrices. Finally, in Section~VI we compute or estimate Lyapunov exponents of special families of matrices arising in problems of functional analysis, combinatorics, and  language theory, and also report numerical results with randomly generated matrices.

Lyapunov exponents of matrices have been studied in the literature in great detail  due to many
applications in probability, ergodic theory, functional analysis, combinatorics, etc.
(see~\cite{O, W, Per, GM, FBS, P5, JPB} and references therein). A special attention has been
paid to the case of nonnegative matrices, i.e., matrices with nonnegative entries~\cite{I, H,  P1}.
The problem of computing or estimating the Lyapunov exponent is known to be extremely hard.
It is even algorithmically undecidable in general~\cite{BT}. Nevertheless, there are several methods
for approximate computation of the Lyapunov exponent that work well in most of practical cases.
These are methods for special families of matrices arising in applications~\cite{FBS, P5, JPB},
for general nonnegative matrices~\cite{Key, GA}, and for general families~\cite{DF1, DF2, BeynL}.
The method proposed in this paper for nonnegative matrices has several advantages compared with those previously known:
1) it produces upper and lower bounds for the Lyapunov exponent $\lambda$ that both converge to
$\lambda$ with a linear rate as the number of iterations grows. So, the Lyapunov exponent is sandwiched between
two values. The rate of convergence is estimated theoretically, but in practice, as we see in numerical examples, it converges much faster. 2) The method works equally well for high dimensions. In examples with randomly generated matrices
of dimension $d \le 50$ it computes $\lambda$ with a relative error less than $1 \%$ within a few iterations.
3) We relax the assumptions on nonnegative matrices imposed in the previous papers on the subject.

The most popular upper bound used in the literature is
\begin{equation}\label{euc}
\frac1k \, \E\ \Bigl\{\, \log \, \|B\| \ \Bigl| \ B \in \cA^k\, \Bigr\}\,
\end{equation}
where $\cA^k$ is the set of all $m^k$ products of matrices of length $k$
(with the corresponding probabilities). For any norm $\|\cdot\|$ this bound converges to $\lambda$ as $k \to \infty$.
Usually one takes the Euclidean norm $\|\cdot\|_2$. As for the lower bounds for nonnegative matrices, the most well-known of them comes from the results of Key~\cite{Key}, which uses the same formula, but with an arbitrary
submultiplicative functional instead of the norm. We shall see that our bounds are closer to the real value of $\lambda$
and have a guaranteed rate of convergence as $k \to \infty$. The theoretical reasons for that are the following:
1) In our bounds, we manage to interchange the Expectation- and Maximum-operations, which results in a smaller upper bound and a larger lower bound.
2) We do not restrict ourself to an a priori fixed functional (or norm), but rather we show how to optimize it over a large family of functionals.

In Section~V we extend this approach to general matrices,
without the nonnegativity assumption, and obtain an upper bound that is better than~(\ref{euc}). We also prove that
such a lower bound for general matrices does not exist.

\section{Operators with a common invariant cone:
 Bounds for the Lyapunov exponent}

Assume that all operators $A_1, \ldots , A_m$ share a common invariant cone~$K \subset \re^d$, which
is supposed to be convex, closed, solid, pointed, and having its apex at the origin.
For any points $x, y \in \re^d$ we write $x \ge y\, $ if $\, x - y \in K$ and
$x - y > 0$ if $x - y \in {\rm int}\, K$.  For an operator~$A$ we write $A \ge 0$
if it leaves the cone $K$ invariant. The same notation are used for the dual cone $K^* =
\{\, v \in \re^d \ | \ \inf_{x \in K}(v, x)\, \ge \, 0\, \}$.

Consider a functional $f : K \to \re_+$. In the sequel we impose the following assumptions on $f$:

1) $f$ is positive, i.e., $f(x) > 0$, whenever $x \ne 0$;

2) $f$ is homogeneous, i.e., $f(tx) = tf(x)$ for any $x \in K, t \in \re_+$.

\smallskip

\noindent Now we define  two values  $F_{\min}$ and $F_{\max}$ for any family
$\cA$ and for any functional $f$:
$$
\begin{array}{lll}
F_{\min}(f, \cA) &  = & \inf\limits_{f(x)  \, = \,
1}\, \E \bigl\{\, \log \, f(Ax)\ \bigl| \ A \in \cA\, \bigr\}\\
{}&{}&{}\\
 F_{\max}(f, \cA) & = & \sup\limits_{f(x)  \, =
\, 1}\, \E \bigl\{\, \log \, f(Ax)\ \bigl| \ A \in \cA\,
\bigr\}\, .
\end{array}
$$
We use the short notation $F_{\min}(f, \cA) = F_{\min}$, and the same with $F_{\max}$,
if the functional $f$ and the family $\cA$ are  fixed.   Denote also $F_{\min}^{(k)} \, = \, \frac1k \, F_{\min}(f, \cA^k)$.
Thus,
$$
F_{\min}^{(k)} \quad =  \quad \frac1k \ \inf\limits_{f(x)  \, = \,
1}\, \E \, \Bigl\{\, \log \, f(Bx)\ \Bigl| \ B \in \cA^k\, \Bigr\},
$$
and similarly with $F_{\max}^{(k)}$. Thus, $F_{\min}^{(k)}$ is the smallest expected value
of the logarithm of the ratio $\frac{f(A_{d_k}\cdots A_{d_1}x)}{f(x)}$ over all $x \in K, x \ne 0$.
Let us first make the following simple observation:
\begin{lemma}\label{l5}
For every natural $k$ and $n$ we have
$$
(k+n)\, F_{\max}^{(k+n)} \ \le \ k\, F_{\max}^{(k)} \ + \ n\, F_{\max}^{(n)}\qquad \mbox{and} \qquad
(k+n)\, F_{\min}^{(k+n)} \ \ge \ k\, F_{\min}^{(k)} \ + \ n\, F_{\min}^{(n)}\, . 
$$
\end{lemma}
\begin{proof} We have
$$
(k+n)\, F_{\max}^{(k+n)} \ = \ \sup\limits_{x\in K}\, \E \bigl\{\, \log \, f(Bx) \, - \, \log \, f(x)\ \bigl| \ B \in \cA^{k+n}\, \bigr\}
\quad =
 $$
$$
\sup\limits_{x\in K}\, \E \bigl\{\, \bigl(\log \, f(B_1B_2x) \, - \, \log \, f(B_2x)\, \bigr) \ + \
\bigl(\, \log \, f(B_2x)\, - \,   \log \, f(x)\, \bigr)\quad \bigl|\quad \ B_1 \in \cA^{k},\ B_2 \in \cA^{n}\, \bigr\}
\ =
$$
$$
\sup\limits_{x\in K}\ \Bigl[ \ \E \bigl\{\, \log \, f(B_1B_2x) \, - \, \log \, f(B_2x)\ \bigl| \ B_1 \in \cA^{k}\, \bigr\} \ + \ \E \bigl\{\, \log \, f(B_2x) \, - \, \log \, f(x)\ \bigl| \ B_2 \in \cA^{n}\, \bigr\} \, \Bigr] \ \le
$$
$$
\sup\limits_{z\in K} \ \E \bigl\{\, \log \, f(B_1z) \, - \, \log \, f(z)\ \bigl| \ B_1 \in \cA^{k}\, \bigr\}   \ + \             \sup\limits_{x\in K}  \ \E \bigl\{\, \log \, f(B_2x) \, - \, \log \, f(x)\ \bigl| \ B_2 \in \cA^{n}\, \bigr\} \, ,
$$
which completes the proof for $F_{\max}^{(k+n)}$. The proof for $F_{\min}^{(k+n)}$ is the same. \end{proof}
\medskip

\begin{corollary}\label{c10}
For an arbitrary functional~$f$ and for every $k$ we have $\, F_{\min}\, \le \, \, F_{\min}^{(k)}\, $
and $\, F_{\max}\, \ge  \, F_{\max}^{(k)}\, $.
\end{corollary}
\begin{lemma}\label{l10}
For an arbitrary functional~$f$ and for any $n$ we have $\ F_{\min}^{(n)}\, \le \, \lambda \, \le \,
F_{\max}^{(n)}$. In particular, $\ F_{\min}\, \le \, \lambda \, \le \,
F_{\max}$.
\end{lemma}
\begin{proof} It suffices to prove that $\ F_{\min}\, \le \, \lambda \, \le \,
F_{\max}$. Then applying this inequality to the family $\cA^n$ and taking into account that
$\lambda(\cA^n) = n \lambda (\cA)$ one obtains $\ F_{\min}^{(n)}\, \le \, \lambda \, \le \,
F_{\max}^{(n)}$.

By the compactness argument it follows that
there are positive constants $C_1, C_2$ such that $C_1 \|x\| \le f(x) \le C_2 \|x\|$. Actually,
these constants are respectively the minimum and the maximum of $f(x)$ on the intersection of the unit sphere
with the cone~$K$.
Applying Corollary~\ref{c10}, we obtain for every $x \in K, x \ne 0$
$$
\frac1k\ \E \Bigl\{\, \log \, \|B\|\, \bigl| \ B \in \cA^{k}\, \Bigr\}\ \ge \
\frac1k\ \E \Bigl\{\, \log \, \|Bx\| \, - \, \log \, \|x\|\ \Bigl| \ B \in \cA^{k}\, \Bigr\}\ \ge
$$
$$
\frac1k\ \E \Bigl\{\, \log \, \frac{C_1}{C_2}\, + \, \log \, f(Bx) \, - \, \log \, f(x)\ \Bigl| \ B \in \cA^{k}\, \Bigr\}\ = \ \frac1k\  \log \, \frac{C_1}{C_2} \ + \ \ F_{\min}^{(k)} \ \ge \ \frac1k\  \log \, \frac{C_1}{C_2} \ + \ F_{\min}\, .
$$
Since $\, \frac1k\, \E \bigl\{\, \log \, \|B\|\, \bigl| \ B \in \cA^{k}\, \bigr\}\, \to \, \lambda\, $ and
$\frac1k\,  \log \, \frac{C_1}{C_2} \, + \, F_{\min} \, \to \, F_{\min}\, $ as $k \to \infty$, we see that
$F_{\min} \le \lambda$.
On the other hand,  the same Corollary~\ref{c10} implies that
$$
\, \frac1k \ \E \Bigl\{\, \log \, f(Bx)\ \Bigl| \ B \, \in \, \cA^k\,
\Bigr\}\quad  \le \quad  F_{\max}
$$
 for every $x \in K$ such that $f(x) = 1$.
Furthermore, for each $x \in {\rm int}\, K$ there is a positive constant $C(x)$ such that for every operator $B$ leaving the  cone $K$ invariant, we have $\|Bx\| \, \ge \, C(x)\|B\|$ (see, for instance,~\cite{P4}). Therefore, $C_2\|B\|\cdot \|x\| \ge f(Bx) \ge C_1C(x)\|B\|$, and
hence
$$
\frac1k\
\E \Bigl\{\, \log \, f(Bx)\ \Bigl| \ B \in \cA^k\,
\Bigr\}\quad \to \quad \lambda\, \qquad \mbox{as} \quad   k \, \to \, \infty \, .
$$
Thus, $\, \lambda \, \le \, F_{\max}$.\end{proof}
\medskip

By Fekete's lemma~\cite{F} for any  sequence of  nonnegative numbers $\{a_k\}_{k \in \n}$
such that \linebreak $(k+n)\, a_{k+n} \, \le \, ka_k + na_n\, , \ k,n \in \n$, the limit $\lim_{k \to \infty}a_k$
exists and equals to $\inf_{k \in \n}a_k$. Similarly, if
$(k+n)\, a_{k+n} \, \ge \, ka_k + na_n\, , \ k,n \in \n$, then  $\lim_{k \to \infty}a_k = \sup_{k \in \n}a_k$.
Applying Lemma~\ref{l5} we see that $\lim_{k \to \infty}F_{\min}^{(k)} = \sup_{k \in \n}F_{\min}^{(k)}$
(denote this limit by $F_{\min}^{(\infty)}$), and $\lim_{k \to \infty}F_{\max}^{(k)} = \inf_{k \in \n}F_{\max}^{(k)}$
(denote this limit by $F_{\max}^{(\infty)}$). Invoking now Lemma~\ref{l10} we obtain the following
\begin{proposition}\label{p10}
For every family $\cA$ and for every functional $f$ we have
\begin{equation}\label{5}
F_{\min}\ \le \ F_{\min}^{(\infty )} \ \le \ \lambda \ \le \ F_{\max}^{(\infty)} \ \le \ F_{\max}\, . 
\end{equation}
\end{proposition}
Thus,  to approximate the Lyapunov exponent one can take an arbitrary functional $f$ and get the values
$F_{\min}$ and $F_{\max}$ as a lower and upper bound respectively. Iterating, one obtains
the bounds $F_{\min}^{(k)}$ and  $F_{\max}^{(k)}$, which, by Corollary~\ref{c10}, are, at least, not worse.
If the two inner inequalities in~(\ref{5}) become equalities, then $F_{\min}^{(k)}$ and  $F_{\max}^{(k)}$
converge from different sides to $\lambda$, which allows us to compute~$\lambda$ with an arbitrary
prescribed accuracy. Sufficient conditions for that are given in Theorem~\ref{th5} below.
Finally, in the ideal case, when $F_{\min} = F_{\max}$, all the inequalities in~(\ref{5}) become equalities.
In this case we get a sharp value of $\lambda$ immediately, just by evaluating $F_{\min}$. Such ``ideal''
functionals $f$ are called {\em invariant}.
\begin{definition}\label{d10}
A functional $f$ is called invariant for a family $\cA$ if
$$
\, - \, f(x)\quad +\quad
\E \ \Bigl\{\, \log \, f(Ax)\ \Bigl| \ A \in \cA\, \Bigr\}\  \equiv \ {\rm const}\,  \qquad  \forall \ x \in K\setminus \{0\}\, .
$$
\end{definition}
Thus, $f$ is invariant if and only if $\, F_{\min} = F_{\max}$. In view of Lemma~\ref{l10} both these
values equal to $\lambda$.
\begin{corollary}\label{c20}
For any invariant functional~$f$ the constant in Definition~\ref{d10} is equal to the Lyapunov exponent $\lambda$
of $\cA$.\end{corollary}

Certainly, invariant functionals do not exist  for all families that have invariant cones. For nonnegative matrices sufficient conditions were obtained in~\cite{P1}, we shall cite that result in Theorem~A (Section~IV).
However, even if an invariant functional exists, it may be very difficult to find or to approximate. Nevertheless, as the following
theorem says, the very existence of an invariant functional guarantees that for an arbitrary functional~$f$
the values $F_{\min}^{(k)}$ and  $F_{\max}^{(k)}$ both converge to $\lambda$ with the linear rate.
\begin{theorem}\label{th5}
For an arbitrary family $\cA$ and for every functional~$f$ we have  $F_{\max}^{(\infty)}\, = \, \lambda$.

If, in addition,
there is an invariant functional for this family, then for every functional~$f$ we have  $F_{\min}^{(\infty)}\, = \, \lambda$.
In this case
$$
 F_{\max}^{(k)} \ - \ F_{\min}^{(k)}\quad \le \quad C\, k^{-1}\, , \qquad k \in \n\, ,
$$ where the constant~$C$ depends only on $\cA$ and on $f$.
\end{theorem}
\begin{proof} There are positive constants $C_1, C_2$ such that $C_1 \|x\| \le f(x) \le C_2 \|x\|$.
Taking the operator norm $\|B\| = \max_{\|x\|=1}\|Bx\|$ and using the fact that the mean of maxima
is bigger than or equal to the maximum of means, we obtain for every $x \in K, \|x\| = 1$
$$
\frac1k\ \E \Bigl\{\, \log \, \|B\|\, \bigl| \ B \in \cA^{k}\, \Bigr\}\ \ge \
\frac1k\ \max_{\|x\|=1} \ \E \Bigl\{\, \log \, \|Bx\|\ \Bigl| \ B \in \cA^{k}\, \Bigr\}\ \ge
$$
$$
\frac1k\ \max_{f(x)=1} \ \E \Bigl\{\, \log \, \frac{C_1}{C_2}\, + \, \log \, f(Bx) \ \Bigl| \ B \in \cA^{k}\, \Bigr\}\ = \ \frac1k\  \log \, \frac{C_1}{C_2} \ + \ F_{\max}^{(k)}\, .
$$
Taking limit as $k \to \infty$, we get $\lambda \ge F_{\max}^{(\infty)}$. Comparing with~(\ref{5}) we see that
$\lambda = F_{\max}^{(\infty)}$.

Let $\tilde f$ be an invariant functional for~$\cA$. By the compactness argument, for an arbitrary functional
 $f$ on $K$ there are positive constants $C_1, C_2$ such that $C_1 \tilde f (x) \le f(x) \le C_2 \tilde f (x), \, x \in K$.
 Therefore,
 $$
F_{\min}^{(k)}(f, \cA) \quad \ge \quad F_{\min}^{(k)}(\tilde f, \cA) \ + \ \frac1k\, \log \, \frac{C_1}{C_2}\quad =
\quad \lambda \ + \ \frac1k\, \log \, \frac{C_1}{C_2}\, .
 $$
In the same way we show that $F_{\max}^{(k)} \, \le \, \lambda \, + \, \frac1k\, \log \, \frac{C_2}{C_1}$, and hence
\begin{equation}\label{diff}
F_{\max}^{(k)} \ - \ F_{\min}^{(k)}\quad \le \quad 2\left( \log\, {C_2} \, - \, \log\, {C_1}\right)\, k^{-1}\, .
\end{equation}

Taking limit as $k \to \infty$, we obtain $\, F_{\min}^{(\infty )}\,  = \, F_{\max}^{(\infty )}$, which
completes the proof.\end{proof}
\medskip

Thus, for every functional $f$ we have $F_{\max}^{(k)} \, \to \, \lambda$.
If the family $\cA$ possesses an invariant functional on the cone~$K$, then for every functional~$f$
the values $F_{\min}^{(k)}$ and  $F_{\max}^{(k)}$ converge from two sides to the Lyapunov exponent~$\lambda$,
and the distance between them decays as $C\, k^{-1}$.
This provides a theoretical opportunity to compute the Lyapunov exponent with a given precision, using an arbitrary functional~$f$. To realize this idea we need to compute the values $F_{\max}^{(k)}$ and $F_{\min}^{(k)}$
for large $k$. Each computation actually requires the resolution of an optimization problem, for which one needs to find a global
optimum of the function $\, \psi_k(x) = \frac1k\ \E \bigl\{\, \log \, \frac{f(Bx)}{f(x)} \, \bigl| \ B \in \cA^{k}\, \bigr\}, $ on the cone~$K$. Therefore, the functional~$f$ should be chosen in a special way, to obtain
the objective function~$\psi_k(x)$ convenient for global minimizing/maximizing.
 In the next section we define
two families of functionals~$f$ (each depending on one $d$-dimensional  parameter), and then, in Section~IV,
we apply them for the cone $K = \re^d_+$ (i.e., for the case of nonnegative matrices). Those functionals
will allow us not only to evaluate the lower and upper bounds for $\lambda$, but also to optimize these bounds
over all values of the parameters. This leads to a fast  algorithm for computing the Lyapunov exponent~$\lambda$
of nonnegative matrices (Section~IV). Even in relatively high dimensions (up to $60$) that algorithm computes the Lyapunov
exponent with a good precision (the relative error is less than $1\%$). The corresponding numerical examples from applications and some results with randomly generated matrices are given in Section~VI. Then, making use of the semidefinite  lifting technique, we partially extend our technique to general matrices, without the nonnegativity assumption. Applying  a special functional~$f$ on the cone of positive semidefinite matrices, we obtain an  upper bound for~$\lambda$, which is, at least, not worse
than the usual upper bound~(\ref{euc}) with the Euclidean norm (Section~V). In practice it is much more efficient,
which is confirmed by numerical examples in Section~VI. As for effective lower bounds, it is shown in Section~V that
they actually do not exist for general matrices. This explains well-known negative theoretical results
on the Lyapunov exponent computation~\cite{BT}.
\begin{remark}\label{r4}
{\em We have seen that for every functional $f$ the value $F_{\max}^{(\infty)}$ actually coincides with the
  Lyapunov exponent. However, for $F_{\min}^{(\infty)}$ this is, in general,  not the case, unless the family
  $\cA$ possesses an invariant functional. The main problem, therefore,
   is the lower bound for the Lyapunov exponent. In Section~V we shall see examples of matrix families that
   have no functionals $f$ such that $F_{\min}^{(k)} \to \lambda$ as $\, k \to \infty$.  That is why the existence of
   an invariant functional is crucial for deriving lower bounds that converge to the Lyapunov exponent. }
\end{remark}

\section{Two special functionals~$\mathbf{f (x)}$}

In this section we define two families of functionals~$f$, which then will be applied to compute the Lyapunov exponent of
nonnegative matrices.

Let $\cA$ be an arbitrary finite family of matrices sharing an invariant cone~$K$.   For every point $x > 0$
consider the functional $f(\cdot)=r_x(\cdot)$  defined on $K$ as follows:
\begin{equation}\label{r}
\, r_x(y)\quad = \quad \min\ \Bigl\{\, r > 0\ \Bigl| \, y \, \le \,
r\, x\ \Bigr\}.
\end{equation}
Geometrically, this functional is a norm on K, whose unit ball is $K\cap (x-K)$. If $r_x(y) \le 1\, ,$ then
$y \le x$, therefore $Ay \le Ax$ for any operator~$A \ge 0$,
and hence $r_x(Ay)\, \le \, r_x(Ax)$. Consequently, for this functional we have
$$
F_{\max}\ = \  \max_{y\, > \, 0}\, \E \ \left\{\ \log \
\frac{r_x(Ay)}{r_x(y)}\quad \Bigl| \quad A \in \cA\, \right\}\ = \
\E \ \Bigl\{\, \log \, r_x(Ax)\ \bigl| \ A \in \cA\ \Bigr\}\, .
$$

Let us denote
$$
\alpha(x)\quad = \quad \E \ \Bigl\{\, \log \, r_x(Ax)\ \bigl| \ A
\in \cA\ \Bigr\} \ ; \qquad \alpha \quad = \quad \inf_{\, x\, >\,
0\, }\ \alpha\, (x)\, .
$$
Similarly we define $\alpha_k(x) = F^{(k)}_{\max}\, $ and $\, \alpha_k \, = \, \inf_{\, x\, >\,
0\, }\ \alpha_k\, (x)$.
 Applying now Lemma~\ref{l10}, we conclude that
$\alpha_k(x) \ge \lambda$ for each $x > 0$ and $k \in \n$, and therefore $\, \alpha_k\, \ge \, \lambda$.
\smallskip

To obtain a lower estimate for  $\lambda$ we take arbitrary $v \in {\rm int}\, K^*$ and
consider the linear functional  $f(x) = (v,x)$. 
Again, this functional is a norm on $K,$ whose unit ball is the intersection of $K$ with a half-space. 
For this functional we have 
$$
F_{\min}\ = \ \inf_{x \, \in \, K}\,  \E \
\left\{\ \log \ \frac{(v, Ax)}{(v,x)}\quad \Bigl| \quad A \in
\cA\, \right\}\, .
$$
Denote
$$
\beta(v)\ = \ \inf_{x \in K}\,  \E \
\left\{\ \log \ \frac{(v, Ax)}{(v,x)}\quad \Bigl| \quad A \in
\cA\, \right\}\ ; \qquad \beta \  = \ \sup_{v \, \in \, {\rm int}\, K^*}\, \beta (v)\, . 
$$
Similarly we define $\beta_k(v) = F^{(k)}_{\min}\, $ and $\, \beta_k \, = \, \inf_{\, v\, >\,
0\, }\ \beta_k\, $.  Lemma~\ref{l10} now implies that $\, \beta_k (v) \,
\le \, \lambda \, $ for every $v \in {\rm int K^*}$ and $k \in \n$.

Thus, we have the following bounds for the Lyapunov exponent of a
matrix possessing an invariant cone:
\begin{equation}\label{E}
\beta_k(v) \quad \le \quad \lambda \quad \le \quad \alpha_k(x).
\end{equation}

In general, they are not easy to compute.
For instance, to evaluate $\beta(v)$ we need to find the global minimum over~$x \in K$ of the function
$$
\E \ \left\{\ \log \ \frac{(v, Ax)}{(v,x)}\quad \Bigl| \quad A \in
\cA\, \right\} \quad = \quad - \log \, (v,x)\ + \  \sum_{j=1}^{m}\ p_j\, \log (v, A_jx)\, .
$$
This function is not convex in~$x$, it is actually quasiconcave, and hence its minimization
may be hard. Nevertheless, we shall see in Section~IV that in case $K = \re^d_+$ the value
$\beta_k(v)$ is not only computable, but can be efficiently optimized over all $v \in {\rm int}\, K^*$, and
the same  is for $\alpha_k(x)$. If we work with a general cone~$K$, then it is more convenient to
apply the linear functional $f(x) = (v,x)$ to get not a lower bound (as $\beta_k(v)$), but the upper one.
Doing so, we write $F_{\max}$ for the functional $f(x) = (v,x)$ and obtain
\begin{equation}\label{gamma1}
\gamma(v) \ = \ \max_{x \in K, \, (v,x) = 1}\  \sum_{j=1}^{m}\ p_j\, \log \, (v, A_j\, x)\, .
\end{equation}
The objective  function $\, \psi(x) = \sum_{j=1}^{m}\ p_j\, \log (v, A_jx)\, $ is concave, and hence its maximum on the
convex set $\{x \in K \ | \ (v,x) = 1\}$ can be efficiently found.
The same can be done for every~$k$:
\begin{equation}\label{gamma2}
\gamma_k(v) \ = \ \max_{x \in K, \, (v,x) = 1}\  \E \
\Bigl\{\ \log \ (v, Bx)\quad \Bigl| \quad B \in
\cA^k\, \Bigr\} \, .
\end{equation}
The shortcoming of this estimate
is that it is very hard to minimize over the set $v \in {\rm int}\, K^*,$ even in the case $K = \re^d_+$.
Nevertheless, choosing appropriate $v$ one can obtain good upper bounds $\gamma_k(v)$ that converge fast to
$\lambda$ as $k \to \infty$. We use this bound in Section~V for approximating  Lyapunov exponents of general sets
of matrices.

\bigskip

\section{The Lyapunov exponent of nonnegative matrices}

We are going to see that in case $K = \re^d_+$, i.e., when all the operators $A_j$
are written by nonnegative matrices,  both estimates $\alpha_k$ and $\beta_k$ are efficiently
computable. We only show here how to evaluate $\alpha$ and $\beta$, since the computation  of
$\alpha_k$ and $\beta_k$ is the same with replacing the family~$\cA$ of $m$
matrices by the family $\cA^k$ of all their $m^k$ products of length~$k$.

We begin with $\alpha$. Let us first note that $r_x(y)\, = \, \max\limits_{i=1, \ldots , d}\, \frac{y_i}{x_i}$. Therefore,
$$
\alpha(x)\quad = \quad \E \ \left\{\ \log \ \Bigl(\ \max_{i=1, \ldots , d}\ \frac{(Ax)_i}{x_i}\ \Bigr)\ \Bigl| \ A
\in \cA\ \right\} \, .
$$
Changing the variables, $\, x_i \, = \, e^{\, u_i}\, , \ u_i \in \re$, we get
$$
\alpha(u)\quad = \quad \E \ \left\{\ \log \ \Bigl(\ \max_{i=1, \ldots , d}\, \sum_{j=1}^d\, a_{ij}\, e^{\, u_j \, - \, u_i}\ \Bigr)\ \Bigl| \ A
\in \cA\ \right\} \, .
$$
Interchanging $\log$ and $\max$ we write
$$
\alpha(u)\quad = \quad \E \ \left\{\ \max_{i=1, \ldots , d}\
\log \ \Bigl( \ \sum_{j=1}^d\ a_{ij}\, e^{\, u_j \, - \, u_i}\, \Bigr)\ \Bigl| \ A
\in \cA\ \right\} \, .
$$
Observe that the function
$\, \log \, \bigl(\, \sum_{j=1}^d\, a_{ij}\, e^{\, u_j \, - \, u_i}\, \bigr)\, $
is convex in $u$. The maximum of convex
functions is convex. Therefore, the value $\alpha$ is a solution of the following convex minimization problem:
\begin{equation}\label{alpha}
\alpha\quad = \quad \inf_{u \, \in \, \re^d}\ \E \ \left\{\ \max_{i=1, \ldots , d}\
\log \ \Bigl( \ \sum_{j=1}^d\, a_{ij}\, e^{\, u_j \, - \, u_i}\, \Bigr)\ \Bigl| \ A
\in \cA\ \right\} \, .
\end{equation}

\medskip

Let us now compute $\beta$. We have
$$
\beta(v)\ = \ \inf_{x \, \ge \, 0\, , \, (v,x)=1}\,  \E \
\left\{\ \log \ (v, Ax) \quad \Bigl| \quad A \in
\cA\, \right\}\, .
$$
Thus,  $\beta(v)$ is the minimal value of the concave function $\, \E \
\left\{\ \log \ (v, Ax) \quad \Bigl| \quad A \in
\cA\, \right\}\, $ on the simplex~$\{x\, \ge \, 0\, , \, (v,x)=1\}$. This minimal value is attained at an extreme point,
i.e., at a  vertex of the simplex. Since its vertices are the vectors of the canonical basis, we have
$$
\beta(v)\quad = \quad \min_{j=1, \ldots , d}\ \left( \ - \, \log \, (v, e_j) \ +  \  \E \
\Bigl\{\ \log \ (v, Ae_j) \ \Bigl| \ A \in
\cA\, \Bigr\}\ \right)\, .
 $$
 Since $\,(v, e_j)\, = \, v_j\, $ and $\,  (v, Ae_j)\, = \, (v, a^j)$, where $a^j$ is the~$j$th column of the matrix~$A$, we obtain
\begin{equation}\label{betapos}
\beta(v)\quad = \quad \min_{j=1, \ldots , d}\ \left( \ -\log \, v_j\  + \ \E \
\Bigl\{\ \log \ (v, a^j) \quad \Bigl| \quad A \in
\cA\, \Bigr\}\ \right)\, .
\end{equation}
For any $j$ and for any $c \in \re$ the set of solutions $ v> 0$ of the inequality
$$
-\log \, v_j\  + \ \E \
\left\{\ \log \ (v, a^j) \quad \Bigl| \quad A \in
\cA\, \right\}\quad \ge \quad c
$$
coincides with the set of solutions of the inequality
$$
{\rm exp}\ \left[\ \E \
\left\{\ \log \ (v, a^j) \quad \Bigl| \quad A \in
\cA\, \right\}  \ \right] \quad = \quad \prod_{i=1}^m \, (v, A_ie_j)^{\, p_i} \quad \ge \quad e^c\, v_j\, ,
$$
which is convex, because the function $\prod_{i=1}^m \, (v, A_ie_j)^{\, p_i}$ is  concave in $v$.
Hence, for any $j$ the function
$\, -\log \, v_j\,  + \,  \E \,
\bigl\{\ \log \ (v, a^j) \ \Bigl| \ A \in
\cA\, \bigr\}\, $ is quasiconcave in $v$, and therefore $\beta(v)$ is quasiconcave as well,
as a minimum of quasiconcave functions. Thus, $\beta$ is the solution of the following
quasiconcave maximization problem:
\begin{equation}\label{beta}
\beta \quad = \quad \sup_{v\, > \, 0}\quad   \min_{j=1, \ldots , d}\quad \left(\  -\log \, v_j\  + \ \ \E
\ \Bigl\{\ \log \ \Bigl( \, \sum_{i=1}^d\, a_{ij}\, v_i \, \Bigr)\quad \Bigl| \quad A \in
\cA\, \Bigr\}\ \right)\, .
\end{equation}
Let us remember that for each $k$ the values  $\alpha_k$ and $\beta_k$ are defined by formulas~(\ref{alpha}) and
(\ref{beta}) multiplied by $\frac1k$ and with replacing $\cA$ by $\cA^k$. Similarly for the values $\alpha_k(x)$ and
$\beta_k(v)$. Thus, for every vectors $x , v > 0$ we have the following inequality:
\begin{equation}\label{k}
 \ \beta_k(v) \quad \le \quad \lambda \quad \le \quad \ \alpha_k(x)\, , \qquad k \in \n\, .
\end{equation}

\bigskip

\subsection{Two conditions for nonnegative matrices}

The bounds $\alpha_k(v)$ and $\beta_k(x)$ are derived for all families of nonnegative matrices,
and inequality~(\ref{k}) always holds. The question, however, is do
they provide really effective estimations for the Lyapunov exponent,  i.e., do they converge to $\lambda$ as
$k \to \infty$ ? It appears that the answer is affirmative, whenever the matrices $A_j$ are not ``too sparse''.
More precisely, the family $\cA$ satisfies the following two assumptions:
\medskip

\textbf{(a)}  there is at least one strictly positive product of matrices from~$\cA$ (with repetitions permitted);

\textbf{(b)} matrices from $\cA$ do not have zero rows nor zero columns.

\medskip

\noindent Note that conditions \textbf{(a)} and \textbf{(b)} are assumed in most of papers studying random products of
nonnegative matrices~(see \cite{W, I, H, Key, P1}). We shall see that condition \textbf{(a)} can always be omitted,
 but the situation with condition~\textbf{(b)} is more complicated.

 Let us recall that  (Theorem~\ref{th5}), if there is an invariant functional for the family~$\cA$, then for \emph{any}
other functional $f$ the bounds $F^{(k)}_{\min}$ and $F^{(k)}_{\max}$ converge linearly to the Lyapunov exponent.
\smallskip

\noindent \textbf{Theorem~A}~\cite{P1}. {\em For every family of nonnegative matrices satisfying conditions
\textbf{(a)} and~\textbf{(b)} there exists an invariant functional on~$\re^d_+$. This functional is, moreover, concave and
monotone on the cone~$K$.}
\smallskip

Applying now Theorem~\ref{th5} to the functionals $f(\cdot) = r_x(\cdot )$ and $f(\cdot ) = (v, \cdot)$, we arrive at
\begin{theorem}\label{th10}
 For every family of nonnegative matrices satisfying
\textbf{(a)} and \textbf{(b)}, and for every vectors $x,v > 0$
we have $\, \beta_k(v) \, \le \,  \lambda \, \le \,   \alpha_k(x)\, , \ k \in \n\, ,$ and
$$
\alpha_k(x)  \ - \ \beta_k(v) \quad \le \quad C\, k^{-1}\, , \qquad k \in \n\, ,
$$
where the constant $C$ depends on $\cA, x$ and $v$.
\end{theorem}
\begin{remark}\label{r7}{\em The constant $C$ can be effectively estimated by entries of the matrices~$A_j$,
see~\cite{P3}.
}
\end{remark}
Since $\beta_k \ge \beta_k (v)$ and $\alpha_k \le \alpha_k(x)$, we see that the estimates
$\beta_k$ and $\alpha_k$ tend to $\lambda$ as well, and $\, \alpha_k  \, - \, \beta_k \, \le \, C\, k^{-1}$.
In general, there is no need to evaluate
the optimal values in problems~(\ref{alpha}) and~(\ref{beta}) with a good precision. To approximate~$\lambda$ it suffices
to find points $x$ and $v$ for which the difference $\alpha_k(x) \, - \, \beta_k(v)$ is small, say, less than $\varepsilon$,
then the Lyapunov exponent $\lambda$ is found with the precision~$\varepsilon$. As we shall see in numerical examples, in practice
the value  $\, \alpha_k  -  \beta_k\, $ decays much faster than $k^{-1}$, and it is enough to take a reasonably
small $k$ (much smaller than $1/\varepsilon$) to compute the Lyapunov exponent $\lambda$ with the precision~$\varepsilon$.
\begin{remark}\label{r10}
{\em Estimates $\alpha_k$ and $\beta_k$ can be extended to any set of matrices sharing a polyhedral invariant cone~$K$.
If the cone $K$ is spanned by vectors $\{h_j\}_{j=1}^{N_1}$ and its dual cone $K^*$
is spanned by vectors $\{g_i\}_{i=1}^{N_2}$, then writing formula for $\alpha(x)$ we replace $\, \frac{(Ax)_i}{x_i}\, $
by $\, \frac{(g_i, Ax)}{(g_i, x)}\, $, and writing formula for $\beta(v)$ we replace $\, \frac{(v, Ae_j)}{v_j}\, $
by $\, \frac{(v, Ah_j)}{(v, h_j)}\, $. Here it is important that both sets of vectors are finite, i.e., that the cone~$K$ is polyhedral.
As we shall see in Theorem~\ref{th30}, for families of matrices with a non-polyhedral invariant cone
 an effective lower bound for~$\lambda$ may not exist at all. In particular, $\beta_k$
is not such a bound any more.
}
\end{remark}

\bigskip

\subsection{Omitting condition (a)}

The lower and upper bounds $\beta_k (v)$ and $\alpha_k(x)$ respectively can be computed for every family of nonnegative matrices.
However, to prove the convergence of these bounds to $\lambda$ we essentially used conditions~\textbf{(a)} and~\textbf{(b)}, because
Theorem~A may fail without them. Moreover, there are simple examples showing that if at least one of the
conditions~\textbf{(a)} or \textbf{(b)} is not fulfilled, then the difference $\alpha_k - \beta_k$ may not vanish as $k \to \infty$,
in which case  our bounds do not provide the Lyapunov exponent computation with a given prescribed accuracy. For every family~$\cA$ condition~\textbf{(b)} can be, of course, checked immediately. Condition~\textbf{(a)} looks
more difficult to verify. Nevertheless is can be checked efficiently as well. The corresponding  algorithm takes~$2md^3$ arithmetic operations, where $d$ is the dimension, and $m$ is the number of matrices~\cite{PV}.
Thus, if both conditions~\textbf{(a)} and \textbf{(b)} are satisfied, then we apply Theorem~\ref{th10} to compute the Lyapunov exponent~$\lambda$. Otherwise, if at least one of them
 fails, one can still use inequality~(\ref{k}), but now there is no guarantee that both parts converge
 to $\lambda$ as $k \to \infty$. For some families this still gives good numerical estimates for $\lambda$
 (as for the binomial matrices from Subsection~VI.1 below). However, there are examples, when
 $\alpha_k$ and $\beta_k$ are too far from each other for all~$k$, and these bounds become useless.
A question arises: is it possible
to obtain effective upper and lower bounds for the Lyapunov exponent (perhaps, different from $\alpha_k$ and $\beta_k$)
without those two conditions?
We do not know the answer for condition~\textbf{(b)}. Is it true that if nonnegative matrices are allowed to have zero rows and columns, then the Lyapunov exponent can be sandwiched between two efficiently computable bounds, whose difference tend to zero?

As for condition~\textbf{(a)}, the answer is affirmative. That condition can be omitted,  with a special modification of the upper bound~$\alpha_k(x)$.
In this subsection we extend our approach to this case, when matrices of the family~$\cA$ do not necessarily have a positive product. To begin with,
we need the following key result proved in~\cite{PV}:
\smallskip

\noindent \textbf{Theorem B} \cite{PV}. {\em If a family of nonnegative matrices~$\cA$ satisfies condition~\textbf{(b)}, but do not have a positive product, then one of the two following cases takes place:

(1) $\cA$ is  reducible;

 (2) there is  a partition of the set $\Omega = \{1, \ldots , d\}$ to $r\ge 2$ sets $\Omega_1, \ldots , \Omega_r $, on which every
 $\qquad $ matrix from~$\cA$ acts as a permutation.

\noindent In the latter case
there exists a product $D$ of matrices from~$\cA$ that has a block-diagonal form:   $r$ strictly positive blocks corresponding to the sets $\Omega_1, \ldots \Omega_r$.}
\smallskip

 Property (1) means that there is a nontrivial subspace of $\re^d$ spanned by several basis vectors, that is invariant for all matrices from $\cA$. Property~(2) means that for every matrix $A \in \cA$ there is a permutation
$\sigma$ of the set $\{1, \ldots , r\}$ such that $AL_k \subset L_{\sigma(k)}, \, k = 1, \ldots , r$, where $\, L_k\, = \, \bigl\{ \, \sum_{i \in \Omega_k}\, t_i e_i \ \bigl| \  t_i \ge 0\, , \, i \in \Omega_k\, \bigr\}\, $ is a
 cone spanned by the vectors $\{e_i \ | \ i \in \Omega_k\}$.
To compute the Lyapunov exponent for a family not satisfying
condition~\textbf{(a)} one needs to consider both cases of Theorem~B.
\medskip

\textbf{Case 1. The family $\mathbf{\cA}$ is  reducible}. In this case there is a permutation of basis vectors, after which
all matrices from $\cA$ take a block upper-triangular form. This permutation can be found by a combinatorial  algorithm
 that takes $O(d^2)$ arithmetic operations (see, for instance,~\cite[Lemma 3.1]{J} and references therein). Now it remains to refer to the main result of the work~\cite{GMO}:
for a family of block upper-triangular matrices the Lyapunov exponent equals to the largest Lyapunov exponent of the blocks.
Hence, in case (1) the problem of computing the Lyapunov exponent is reduced to several analogous problems in smaller dimensions.
\medskip

\textbf{Case 2. There is  a partition of the set $\mathbf{\Omega}$, on which every
matrix from~${\mathbf \cA}$ acts as a permutation. } We first show that in this case, we can restrict our attention to the set $L = \cup_{i=1}^{\, r}\, L_i$.
This is a union of faces~$L_i$ of the positive orthant~$\re^d_+$,
corresponding to the sets $\Omega_i$ of the partition. Clearly, $A_jL \subset L$ for each $j = 1, \ldots , m$.
Consider an arbitrary functional~$f$ on the set~$L$, which is positive ($f(x)>0\, , \, x \in L, x\ne 0$)
and homogeneous ($f(tx) = tf(x), \, x \in L, \, t \ge 0$). For this functional we define $F_{\min}, F_{\max}, F_{\min}^{(k)},
F_{\min}^{(k)}$ in the same way as in Section~II. Then we establish the following  analogue  of Lemma~\ref{l10}:
\begin{lemma}\label{l20}
If a family $\cA$ is irreducible, then for an arbitrary functional~$f$ on $L$ and for each $n$
we have $\ F_{\min}^{(n)}\, \le \, \lambda \, \le \,
F_{\max}^{(n)}$. In particular, $\ F_{\min}\, \le \, \lambda \, \le \,
F_{\max}$.
\end{lemma}
\begin{proof} The proof is literally the same as the proof of Lemma~\ref{l10} 
with only one exception: in order to prove
 that $F_{\max} \ge \lambda$ we need to show that there exists a vector $x \in L$, for which
\begin{equation}\label{need}
\frac1k\
\E \, \Bigl\{\  \log \ \bigl\|\, B\, x\, \bigr\|\quad  \Bigl| \quad B \, \in \, \cA^k\,
\Bigr\}\ \to \ \lambda\qquad \ \mbox{as} \ k \ \to \ \infty\, .
\end{equation}
 In the proof of Lemma~\ref{l10}
we established the existence of such a point $x$ in the cone $K$, now we need this point in the set~$L$.
To this end we apply the main result of the work~\cite{Hong}: if a family~$\cA$ is irreducible and
its matrices have no zero columns and rows, then every nonzero vector $x \in \re^d_+$ satisfies~(\ref{need}).
Thus, an arbitrary   $x \in L\, , \, x \ne 0$ suffices.
 The remainder of the proof is the same as for Lemma~\ref{l10}.\end{proof}
\medskip

From~\cite[theorem~4]{P3} it follows that for a family of matrices
satisfying assumptions of the case~(2) of Theorem~B there exists an invariant functional~$\tilde f$
on $L$, for which $F_{\min} = F_{\max} = \lambda$. Now, precisely as in the proof of Theorem~\ref{th5},
we conclude that for every functional $f$ on $L$ one has
\begin{equation}\label{k1}
 F_{\max}^{(k)} \ - \ F_{\min}^{(k)}\quad \le \quad C\, k^{-1}\, , \qquad k \in \n\, ,
\end{equation}
where the constant~$C$ depends only on $\cA$ and on $f$. To estimate the Lyapunov exponent 
it remains only to choose any convenient functional~$f$ in order to compute
 the values $F_{\min}^{(k)}$ and  $F_{\max}^{(k)}$.
\smallskip

For $F_{\min}$ we again choose $f(x) = (v,x)$ with arbitrary $v > 0$.
It appears that for this functional we again have the equality $F_{\min} = \beta(v)$, where $\beta(v)$
is defined by~(\ref{betapos}). This is not obvious, because now we take minimum not over
the whole set $\re^d_+$, but over a much narrower set~$L$.  We have
$$
F_{\min}\ = \ \min_{n=1, \ldots , r}\quad \inf_{x \, \in \, L_n\, , \, (v,x)=1}\,  \E \
\left\{\ \log \ (v, Ax) \quad \Bigl| \quad A \in
\cA\, \right\}\, .
$$
Since the minimum of a concave function $\E \, \bigl\{\, \log \, (v, Ax) \ \bigl| \ A \in
\cA\, \bigr\}\,$ on the simplex~$\{x\, \in \, L_n\, , \, (v,x)=1\}$ is attained at an extreme point,
i.e., at a basis vector, we have
$$
F_{\min}\ = \ \min_{n=1, \ldots , r}\, \min_{j \, \in \, \Omega_n\, }\,  \E \,
\left\{\, \log \, \frac{(v, Ae_j)}{(v, e_j)} \ \Bigl| \ A \in
\cA\, \right\}\ = \ \min_{j = 1, \ldots , d}\ \E \,
\left\{\, \log \, \frac{(v, a^j)}{v_j} \ \Bigl| \ A \in
\cA\, \right\}\ = \  \beta(v). 
$$
\smallskip

For $F_{\max}$ we again take the functional $\, r_x(y) = \max_{i=1, \ldots,d} \frac{y_i}{x_i}\, $, where $x > 0$.
For every $n = 1, \ldots , r$ we have $\max\limits_{y \in L_n\, , \, r_x(y) = 1}\, r_x(Ay)\, = \,
\max\limits_{i \in \Omega_{\sigma(n)}}\frac{(Ax)_i}{x_i}$, where $\sigma$ is our permutation. Whence,
$$
F_{\max}\quad = \quad \max_{n = 1, \ldots , r}\ \E \
\Bigl\{ \ \log \ \max_{i \in \Omega_{\sigma(n)}}\frac{(Ax)_i}{x_i} \ \Bigl| \
A \in \cA \ \Bigr\}\, .
$$
This value will be denoted by $\tilde \alpha (x)$.  Interchanging $\max$ and $\log$ and writing
$j = \sigma(n)$, we obtain
\begin{equation}\label{talpha}
\tilde \alpha (x)\quad = \quad \max_{j = 1, \ldots , r}\ \E \
\Bigl\{ \ \max_{i \in \Omega_{j}} \ \log \ \frac{(Ax)_i}{x_i} \ \Bigl| \
A \in \cA \ \Bigr\}\, , 
\end{equation}
and
\begin{equation}\label{talphak}
\tilde \alpha_k (x)\quad = \quad \frac1k\ \max_{j = 1, \ldots , r}\ \E \
\Bigl\{ \ \max_{i \in \Omega_{j}} \ \log \ \frac{(Bx)_i}{x_i} \ \Bigl| \
B \in \cA^k \ \Bigr\}\, .
\end{equation}
Applying now (\ref{k1}) and~Lemma~\ref{l20} we obtain
\begin{theorem}\label{th20}
 For every family of nonnegative matrices possessing property (2) from Theorem~\ref{th20},
 and for every vectors $x,v > 0$ we have
 $$
 \beta_k(v) \ \le \ \lambda \ \le \ \tilde \alpha_k(x)\, , \qquad k \in \n \, ,
 $$
and
$$
\tilde \alpha_k(x)  \ - \ \beta_k(v) \quad \le \quad C\, k^{-1}\, , \qquad k \in \n\, ,
$$
where the constant $C$ depends on $\cA, x$ and $v$.\end{theorem}

Let us note that the values $\beta_k = \sup_{v > 0}\beta_k(v)\, $ and
$\tilde \alpha_k = \inf_{x > 0}\tilde \alpha_k(x)\, $ can be efficiently found
by using convex programming, because the function $\beta_k(v)$ is quasiconcave in~$v$, and
$\tilde \alpha(x)$ is convex in $u$, where $x_i = e^{u_i}\, , \, i = 1, \ldots , d$.
The first assertion is proved, the proof of the second one is the same as for the function~$\alpha_k(x)$.
Thus, both bounds from Theorem~\ref{th20} can be optimized by parameters. In practice this significantly improves
the rate of convergence to $\lambda$.
\smallskip

We see that for every family $\cA$ of nonnegative matrices satisfying condition~\textbf{(b)}
there is an efficient method of computing the Lyapunov exponent. For families satisfying condition~\textbf{(a)}
the method is given by Theorem~\ref{th10}, for a reducible family (the case~(1) of Theorem~B) the problem
is equivalent to several problems of smaller dimensions, for a family of the case~(2) of Theorem~B
the method is given by Theorem~\ref{th20}. The optimal parameters $v$ and~$x$ in the bounds can be
effectively found by convex programming. This covers all families of nonnegative matrices without zero rows and
columns.
\begin{corollary}\label{c30}
For every family $\cA$ of nonnegative matrices that have neither zero columns nor zero rows,
and for every vector $v >0$ we have $\, \lambda \, \ge \, \beta_k(v)\, $ and
$\, \lambda \, - \, \beta_k(v) \, \le \, C\, k^{-1}\, , \, k \in \n$.
\end{corollary}

\bigskip

\subsection{Is it possible to avoid condition~(b) ?}

According to Theorem~\ref{th5} for every family of nonnegative matrices, without any additional assumptions,
the upper bound $\alpha_k$ converges to the Lyapunov exponent $\lambda$ as $k \to \infty$. For the lower bound $\beta_k$
this is also true, provided the family~$\cA$ is irreducible and condition~\textbf{(b)} is fulfilled.
Moreover, the irreducibility assumption is not essential, because if the family is reducible, then the
problem of the Lyapunov exponent computation is equivalent to several problems of smaller dimension. In turn,
condition~\textbf{(b)} is, in general, unavoidable. Indeed, if one of the matrices have a zero column, then
the lower bound becomes trivial: $\beta_k = -\infty$ for all $k$.
In some cases (but not always!) this difficulty can be overcome by special tricks. Let us describe two of them:
\smallskip

1) {\em Considering the transposed family}. If all matrices of the family $\cA$ have no zero columns (but may have zero rows),
then the value $\beta_k$ is different from $-\infty$ for each $k$, and may converge to $\lambda$ as $k \to \infty$, although this convergence is not theoretically guaranteed. Hence, if the matrices of~$\cA$ have zero columns, then one can compute
$\beta_k$ for the family $\cA^* = \{A_1^*, \ldots , A_m^*\}$. In many cases this leads to good lower bounds.

\smallskip

2) {\em The modified estimate $\beta_k(v)$}. In some cases one can consider the following modified
value for a vector $v \in \re^d_+, v \ne 0$:
\begin{equation}\label{betapos-mod}
\tilde \beta_k(v)\quad = \quad \min_{j=1, \ldots , d\, , \, v_j > 0}\ \left( \ -\log \, v_j\  + \ \E \
\Bigl\{\ \log \ (v, b^j) \quad \Bigl| \quad B \in
\cA^k\, \Bigr\}\ \right)\, ,
\end{equation}
where $b^j$ is the $j$th column of the matrix~$B$.
In contrast to the value $\beta_k(v)$, here the vector $v$ may have zero components, and the minimum is taken
over its nonzero components. It is shown easily that for every $v$ this is a lower bound for $\lambda$.
If we choose $v$ so that it has zeros at all positions corresponding to zero columns of matrices
from $\cA^k$, then $\beta_k(v)$ may converge to $\lambda$. We illustrate this trick in Section~VI by a numerical example and apply it for estimating the Lyapunov exponent of large sparse matrices arising in one problem of the language theory.

Let us stress that for both these tricks there are corresponding counterexamples, when they do not work.
In general, if the matrices have zero columns/rows, we do not know any satisfactory lower bound for $\lambda$.
Finding such a bound can be considered as a challenging open problem.

\bigskip

\section{The Lyapunov exponent of general matrices}

In this section we present an efficient upper bound for the Lyapunov exponent of general matrices
(not necessarily nonnegative), and
argue that such a lower bound does not exist.

Let us first assume that all matrices $A_1, \ldots , A_m$ share a common invariant cone~$K$.
The Lyapunov exponent~$\lambda$ is defined as the limit of the value $\, \frac1k \, \E \bigl\{\, \log \, \sup\limits_{x \in K, \|x\| \le 1} \, \|Bx\|\ \bigl| \, B \in \cA^k\, \bigr\}\, $ as $\, k \to \infty$. For every $k$ this value is an upper bound
for~$\lambda$. More generally, for any positive homogeneous functional on $K$ the value
\begin{equation}\label{up1}
\frac1k \ \E \ \Bigl\{\ \log \, \sup_{x \in K, f(x) \le 1} \, f(Bx)\ \Bigl| \ B \in \cA^k\ \Bigr\}
\end{equation}
is an upper bound for $\lambda$, which by Theorem~\ref{th5} converges to~$\lambda$ as $k \to \infty$.
This upper bound is usually applied in the literature to estimate the Lyapunov exponent. The most popular
functional $f$ here is the Euclidean norm. On the other hand, the value
\begin{equation}\label{up2}
F_{\max}^{(k)}\quad = \quad \frac1k \  \, \sup_{x \in K, f(x) \le 1} \ \E \Bigl\{\ \log \, f(Bx)\ \Bigl| \ B \in \cA^k\, \Bigr\}
\end{equation}
is, at least, not bigger (in most cases, actually, much smaller) than~(\ref{up1}), because the maximum of means does not
exceed the mean of maxima. Therefore, the upper bound~$F_{\max}^{(k)}$ is closer to~$\lambda$ than~(\ref{up1}). However, its computation may face serious difficulties, because
it involves finding the maximal value of the function $\, \psi_k(x) = \, \frac 1k\,  \E \bigl\{\, \log \, f(Bx)\ \bigl| \ B \in \cA^k\, \bigr\}\, $ on the set $\, {x \in K, f(x) \le 1}$. Basically, such a maximization problem can be efficiently solved
only in the case, when the function $\psi_k(x)$ is concave, or quasiconcave. Unfortunately, it does not possess this property for most of norms~$f(x)$, including the Euclidean norm. That is why we use the linear functional $f(x) = (v, x)$, for which
the function $\, \psi_k(x) = \, \frac 1k\, \E \bigl\{\, \log \, (v, Bx)\ \bigl| \ B \in \cA^k\, \bigr\}\, $
is concave, and can be effectively maximized. Its maximum gives the upper bound $\gamma_k(v)$ defined in~(\ref{gamma1})
and in~(\ref{gamma2}).  As we have already mentioned, the shortcoming of this upper bound is that the function $\gamma_k(v)$
is not convex, and one is not able to efficiently minimize it over $v \in K^*$. Nevertheless, it is still possible to pick an arbitrary $v,$ hoping to obtain a good bound for~$\lambda.$

To extend this approach to general matrices, without the common invariant cone assumption,  we apply the so-called {\em semidefinite lifting}.
Let us have a family $\cA = \{A_1, \ldots , A_m\}$. To each matrix $A_j$ we associate an operator $\tilde A_j$
on the $\frac{d^2 + d}{2}\, $- dimensional space $\cM_d$ of symmetric $d\times d$-matrices defined as follows:
$$
\tilde A_j \, X \quad = \quad A_j \, X\, A_j^*\, , \qquad X \, \in\, \cM_d\, .
$$
We thus obtain the family $\tilde \cA = \{\tilde A_1, \ldots , \tilde A_m\}$. All operators $\tilde A_j$ now share
a common invariant cone: the cone $\cK_d$ of symmetric positive semidefinite  $d\times d$-matrices.
Let us recall that $\cK^* = \cK$ and that the scalar product in the space $\cM_d$ is defined as $(X, Y) = {\rm tr} (XY)$.
This is shown easily that $\lambda (\tilde \cA) \, = \, 2\, \lambda ( \cA)$. Therefore, we can apply the upper bound~$\gamma_k(v)$ to the family $\tilde \cA$ and get the following upper bound for $\lambda (\cA)$
\begin{equation}\label{Kr1}
\Gamma_k(V)\quad = \quad \frac{1}{2k}\quad \sup_{X \, \succeq\,  0, \ {\rm tr} \, (VX) = 1} \quad  \E \ \Bigl\{ \, \log\ {\rm tr} \, (VBXB^*)   \ \Bigl| \ B \, \in \, \cA^k\ \Bigr\}\, ,
\end{equation}
where $X \, \succeq\,  0$ means that the matrix $X$ is positive semidefinite.
For every positive definite matrix $V$ the value $\Gamma_k(V)$ is an upper bound for $\lambda(\cA)$, and it converges
to it as $k \to \infty$. To compute $\Gamma_k(V)$ one needs to find the maximum of a smooth concave function
on the intersection of the cone $\cK_d$ with a hyperplane $\{X \in \cM_d \ | \ {\rm tr} \, (VX) = 1 \}$.
This problem can be efficiently solved  by standard tools of semidefinite programming (SDP, see, for instance,~\cite{VB}).
In many cases the most natural choice is to take $V=I$ (the identity matrix), which yields the following upper bound:
\begin{equation}\label{Kr2}
\Gamma_k(I)\quad = \quad \frac{1}{2k}\quad \sup_{X \, \succeq\,  0, \ {\rm tr} \, (X) = 1} \quad  \E \ \Bigl\{ \, \log\ {\rm tr} \, (BXB^*)   \ \Bigl| \ B \, \in \, \cA^k\ \Bigr\}\, ,
\end{equation}
Let us show that this upper bound is better that the standard upper bound with the Euclidean norm.
\begin{proposition}\label{p20}
For every matrix family we have $\, \Gamma_k(I)\, \le \, \frac1k\, \E \, \bigl\{ \log \|B\|_2 \ \bigl| \ B \in \cA^k\, \bigr\}\, , \ k \in \n$.
\end{proposition}
\begin{proof} We have
\begin{equation}\label{evkl}
2k\, \Gamma_k(I) \quad \le \quad \E  \ \Bigl\{\log\ \sup_{X \, \succeq\,  0, \ {\rm tr} \, (X) = 1}   \  {\rm tr} \, (BXB^*)   \ \Bigl| \ B \, \in \, \cA^k\ \Bigr\}\, .
\end{equation}
The supremum of the linear function $g(X) = {\rm tr} \, (BXB^*)$ is attained at an extreme point of the set
 $X \, \succeq\,  0, \ {\rm tr} \, (X) = 1$, i.e., at a rank one matrix $X = YY^*$, where $Y \in \re^d$ is a vector.
 Since ${\rm tr}\, (YY^*)\, = \, \|Y\|_2^2\, $ and $\, {\rm tr}\, (BYY^*B^*)\, = \, \|BY\|_2^2$, we see that
  $\sup_{X \, \succeq\,  0, \ {\rm tr} \, (X) = 1}   \  {\rm tr} \, (BXB^*)\ = \
\sup_{Y\in \re^d, \ \|Y\| = 1}   \  \|BY\|_2^2 \ = \ \|B\|_2^2$, and whence the right hand side of inequality~(\ref{evkl})
equals to $\, 2\, \bigl\{ \log \|B\|_2 \ \bigl| \ B \in \cA^k\, \bigr\}$.\end{proof}
\medskip

In Section~VI we present numerical results with randomly generated matrices showing that the bound
$\Gamma_k(I)$ can indeed be much closer to $\lambda$ than the standard bound with the Euclidean norm.
\smallskip

Now let us explain why there is no good lower bound for general families of matrices, even if they are non-factorable
(do not have nontrivial common invariant subspaces in~$\re^d$) and share a common invariant cone (not a polyhedral cone, when we have the lower bound~$\beta_k(v)$, see Remark~\ref{r10}).
To avoid any trouble with the case $\lambda = -\infty$, we formulate the result for lower bounds of the {\em Lyapunov radius} $\, \rho \, = \, e^{\, \lambda}$.
\begin{theorem}\label{th30}
a) There is a non-factorable pair $\bar \cA = \{A_1, A_2\}$ of $\, 2\times 2$-matrices such that for every
function~$\varphi$ on the set of all non-factorable pairs of $2\times 2$-matrices, which is continuous
at the point $\bar \cA$ and $\varphi (\cA) \le \rho (\cA)$ at any point $\cA$, we have $\, \varphi (\bar \cA) \, \le \, \rho(\bar \cA) \, -  \, 1$.

b) There is a non-factorable pair $\bar \cA = \{A_1, A_2\}$ of $\, 3\times 3$-matrices sharing an invariant cone~$K$ such that for every
function $\varphi$ on the set of all non-factorable pairs of $3\times 3$-matrices sharing the invariant cone~$K$, which is continuous at the point $\bar \cA$ and $\varphi (\cA) \le \rho (\cA)$ at any point $\cA$, we have
$\, \varphi (\bar \cA) \, \le \, \rho(\bar \cA) \, - \, 1$.
\end{theorem}
\begin{proof} a) Consider a pair $\cB = \{B_1, B_2\}$, where $B_1$ is a rotation of the plane by the angle~$\frac{\pi}{3}$
about the origin, and $B_2$ is the orthogonal projection onto the $OX$ axis. It is easily shown that
the Euclidean norm of every product of length~$k$ of matrices $B_1, B_2$ is at least~$2^{-k}$.
Hence, $\rho(\cB) \ge 1/2$. Let now $B_{1, n}$ be a  rotation of the plane by the angle $\frac{\pi}{2}\cdot \frac{2n+1}{3n+1}$.
about the origin, and $\cB_n = \{B_{1, n}, B_2\} $. Clearly, $\cB_n \to \cB$ as $n \to \infty$.
On the other hand, $B_2B_{1, n}^{3n+1}B_2 = 0$, and hence $\rho(\cB_n) = 0$ for each $n \in \n$.
Therefore, if $\varphi (\cB_n) \le \rho (\cB_n) = 0$ for all~$n$, then by continuity $\varphi (\cB) \le 0$.
Thus, $\rho (\cB) - \varphi (\cB) \ge 1/2$. Now it remains to take $\bar \cA = \{2B_1, 2B_2\}$.

b) It suffices to take $\bar \cA = \{4\tilde B_1, 4 \tilde B_2\}$, where $B_1, B_2$ are the $2\times 2$-matrices from the
proof of part~(a),  $\tilde B_i$ is the semidefinite lifting of the matrix
$B_i$. Matrices of this pair is three-dimensional, and they share a positive definite cone $\cK_3$.\end{proof}
\medskip

Thus, there are pairs of matrices, whose Lyapunov exponent cannot be well-approximated by a continuous lower bound.
This means that there is no algorithm, whose output continuously depends on the data, which for an
arbitrary pair of matrices computes the Lyapunov exponent with a given precision. It should also be mentioned that, as it was shown by Blondel and Tsitsiklis (1997),
the problem of computing the Lyapunov exponent  for matrices with integer entries is algorithmically undecidable~\cite{BT}.
\begin{remark}\label{r30}
{\em The statement (b) of Theorem~\ref{th30} does not hold for $2\times 2$-matrices, because every cone in~$\re^2$
is polyhedral (see Remark~\ref{r10}). }
\end{remark}

\begin{remark}\label{r20}
{\em The phenomenon that the Lyapunov exponent has good upper bounds, but does not have lower ones
is explained by the fact that this value is an upper semicontinuous function on the families of matrices,
but not lower semicontinuous. In the proof we used a pair of matrices, where   the Lyapunov exponent
is not lower semicontinuous. }
\end{remark}
\bigskip

\section{Numerical examples}

In this section, we show the efficiency of our method on several examples. We first analyze matrices drawn from an application in combinatorics, for which the exact value is known.   Then, in Subsection~VI.2, we study an application in functional analysis.  We then provide estimates for the Lyapunov exponent of some matrices arising in the language theory (Subsection~VI.3).
Finally, in Subsection~VI.4, we analyze diverse kinds of randomly generated matrices. As mentioned above, all optimization problems that we need to solve for computing these estimates are convex unconstrained problems.  We apply
 standard tools of convex optimization, namely, the matlab function \emph{fminunc}, which uses a quasi-Newton procedure (the so-called \emph{BFGS}-scheme), together with a cubic line search procedure.
All computations took a few minutes on a standard desktop pc.

As a general observation, our upper bound generally converges faster than the Euclidean bound (that is, the bound obtained from~(\ref{euc}) with the Euclidean norm) towards the real value in the first steps of the algorithm.  This allows us to have a fair upper estimate without having to compute too large products of matrices.
On the other hand, for our lower bound, not only it also converges faster than the previously available lower estimates (see~\cite{Key}), but in addition, generally these other lower bounds do not allow to reach a satisfactory accuracy at all. Even though they converge asymptotically towards the exact value, they appear to be often relatively far from it for the largest values of $k$ that a classical computer can handle (say, $k = 14$ or so).
We graphically present our results in terms of the Lyapunov radius $\, \rho\, =\, {\rm exp}\, (\lambda),$ because in most applications, it is the ``physical'' quantity that one wants to compute.
\smallskip

\subsection{The binomial triangle and generalizations}

 Recently, the
question of the density of ones in the $n$th row of Pascal's triangle has been studied in
number theory. This question is equivalent to the number of odd coefficients in the polynomial $(1+x)^n.$  In \cite{FBS}, the
authors generalize the question to the study of odd coefficients for the polynomial $(1+\cdots+x^m)^n.$  They show that for any $m \in \mathbb N,$ there exists a set of matrices $\Sigma_m$ such that for large $n,$ with probability one, the proportion of odd coefficients is given by the Lyapunov radius of $\Sigma_m.$  As an example,
we represent hereunder $\Sigma_6:$
\begin{equation}
\Sigma_6= \left \{ \begin{pmatrix}
 1& 0& 1& 2& 0& 0\\
 0& 0& 0& 0& 0& 0\\
 0& 0& 0& 0& 1& 2\\
 0& 2& 1& 0& 1& 0\\
 0& 0& 0& 0& 0& 0\\
 0& 0& 0& 0& 0& 0\\\end{pmatrix},
\begin{pmatrix}
 0& 0& 0& 2& 1& 0\\
 1& 0& 0& 0& 0& 0\\
 1& 0& 0& 0& 0& 2\\
 0& 2& 1& 0& 0& 0\\
 0& 0& 1& 0& 0& 0\\
 0& 0& 0& 0& 1& 0\\\end{pmatrix} \right \}.
 \end{equation}

\begin{figure}
\centering
\begin{tabular}{cc}
\includegraphics[scale = 0.35]{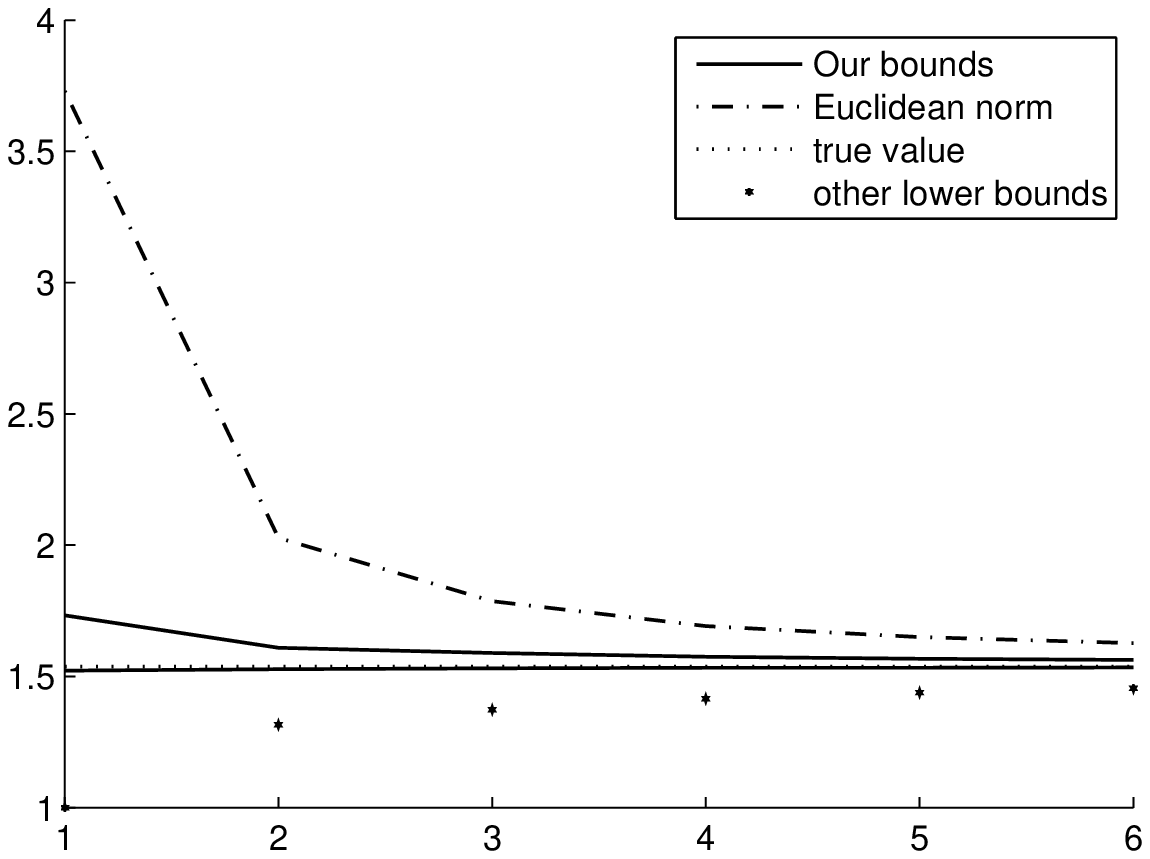}&
\includegraphics[scale = 0.3]{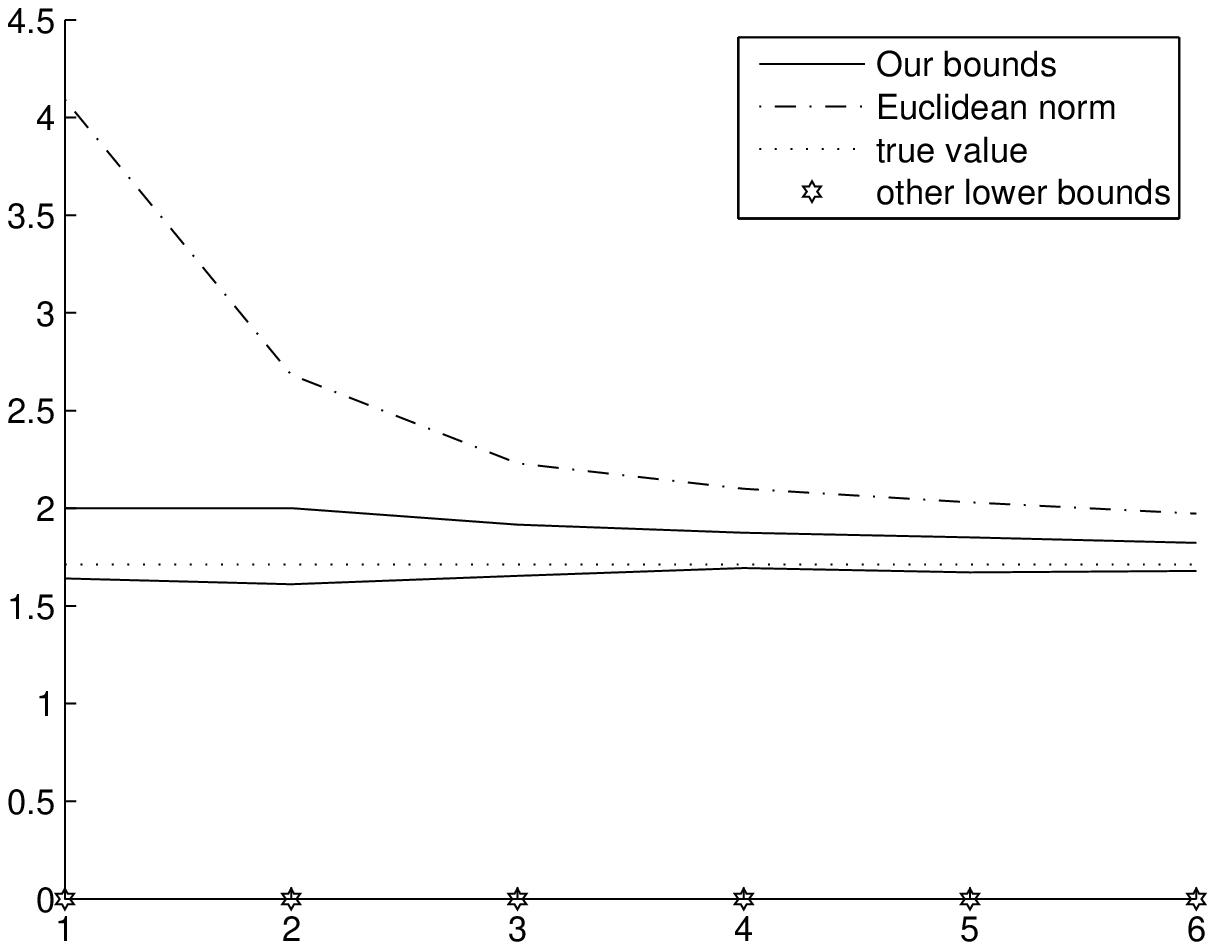} \\(a)&(b)
\end{tabular}
\caption{\footnotesize{Evolution of the different bounds for $\Sigma_2$ (a) and $\Sigma_6$ (b).}} \label{fig-finch}
\end{figure}
 Moreover, it is shown in \cite{FBS} that for these particular sets of matrices, it is possible to compute the Lyapunov exponent exactly.  For this reason, we start this section on numerical examples with this application, for which it is possible to refer to the exact solution.  Figure~\ref{fig-finch} shows how our estimates evolve in comparison to the exact value, and to the upper bound obtained from~(\ref{euc}) with an arbitrary norm. We chose the Euclidean norm, as it most often performs best in practice. Observe that the family $\Sigma_6$ does not satisfy condition~\textbf{(b)} (the first matrix has three zero rows), hence the convergence  of $\beta_k$ towards $\lambda$ is not guaranteed theoretically. Nevertheless, both $\beta_k$ and $\alpha_k$ converge rapidly towards the exact value of~$\lambda$.  Moreover, as one can see on figure \ref{fig-finch} (b), it may happen that the previously known lower bounds only give the trivial zero lower bound, while our lower bound rapidly converges towards the exact value.

\bigskip
\subsection{The regularity of de Rham curves}

\begin{figure}
\centering
\begin{tabular}{ccc}
\includegraphics[scale = 0.27]{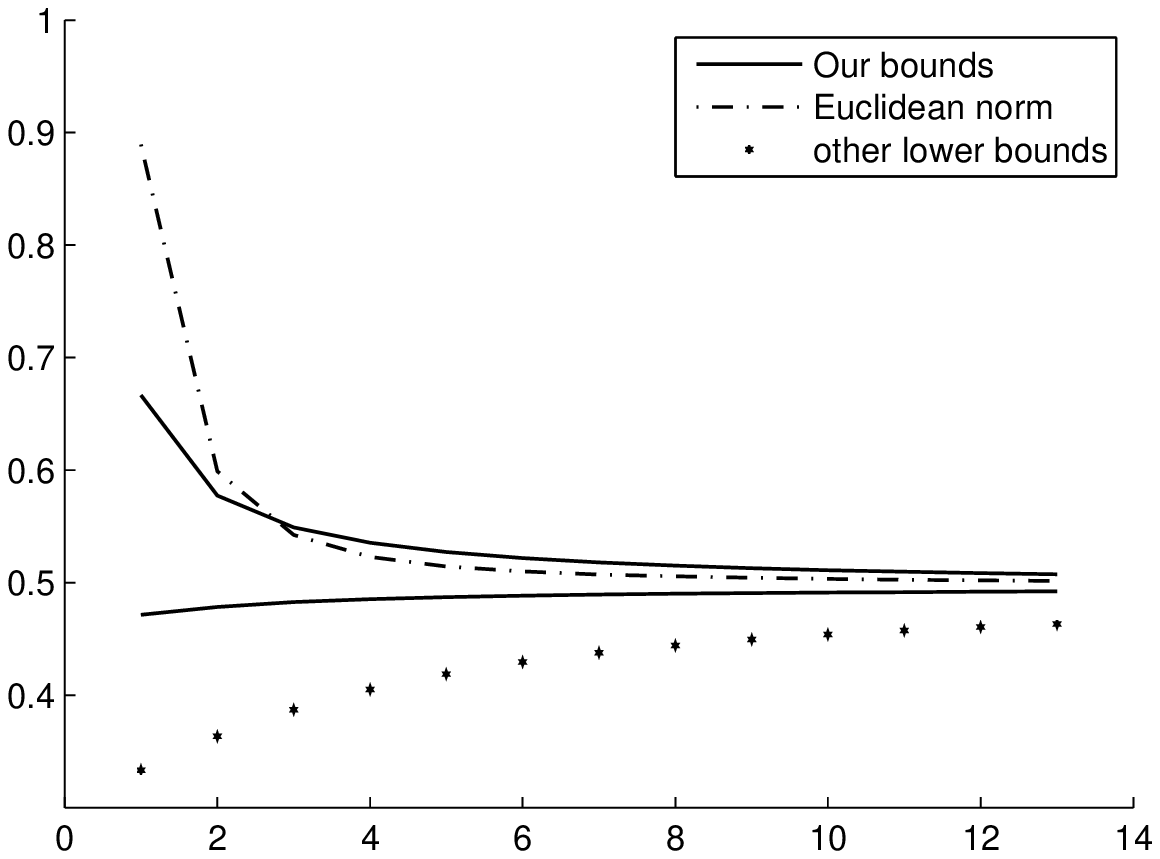}&
\includegraphics[scale = 0.27]{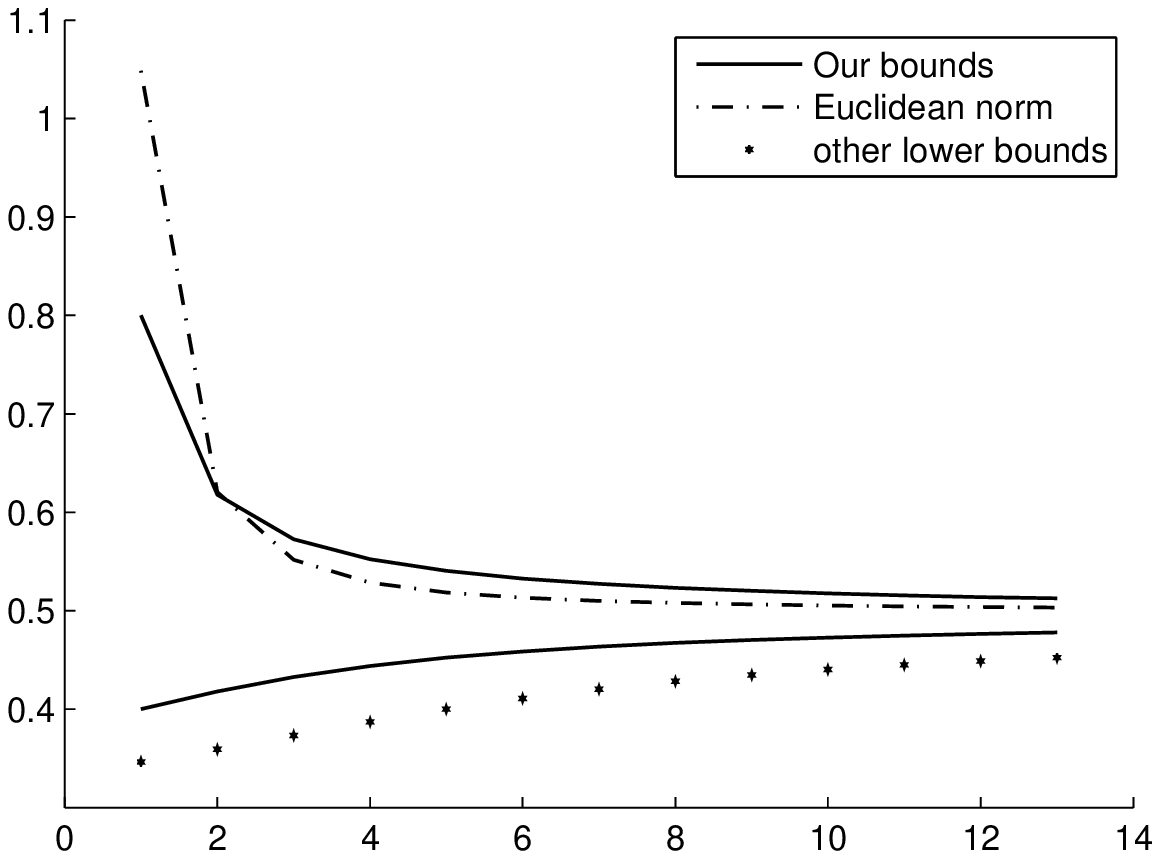}&
\includegraphics[scale = 0.27]{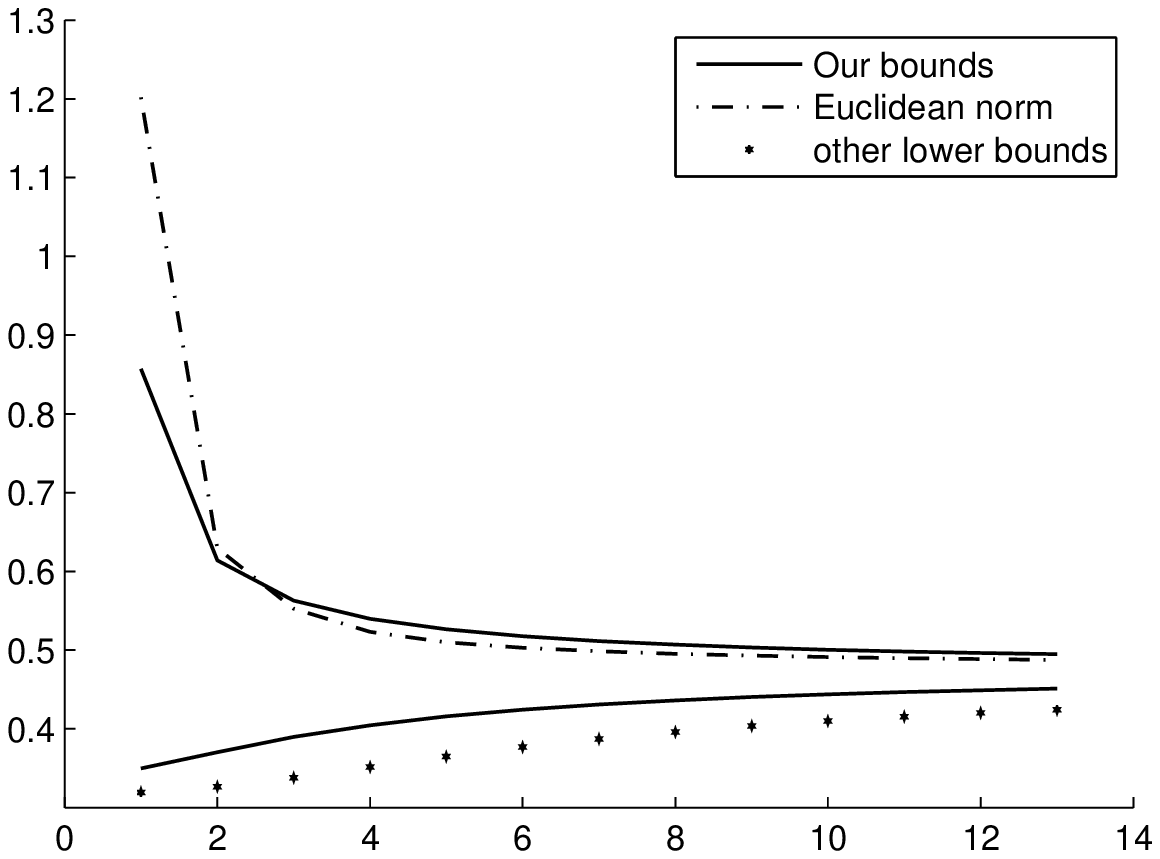} \\(a)&(b)&(c)
\end{tabular}
\caption{\footnotesize{Evolution of the different bounds for $\cA_{1/3}$ (a), $\cA_{1/5}$ (b), and $\cA_{1/7}$ (c).}} \label{fig-derham}\end{figure}

An important application of the joint spectral characteristics of matrices is the exponents of global
and local regularity of solutions of functional equations. In particular, they measure the regularity of fractal curves, refinable functions, and wavelets.  To a given  refinable function, one can associate a set of matrices so that the Lyapunov exponent of this set is equal to  the local regularity of the function at almost all points (in Lebesgue measure) of its domain.  We consider the simplest example of refinable functions, the so-called de Rham curves, which are obtained from an arbitrary flat polygon by successive cutting off its angles, when each side is divided by the same ratio
 $\, \omega: (1-2\omega): \omega$, where $\omega \in \bigl(0, \frac12 \bigr)$ is a given parameter.
The corresponding matrices of de Rham curve are given by
$$
\cA_{\, \omega} \ = \ \left\{ \left(
\begin{array}{cc}
\omega & 0\\
 \omega &  1-2\omega
\end{array}
\right)  , \ \left(
\begin{array}{cc}
1-2\omega & \omega \\
0 & \omega
\end{array}
\right) \right\}. $$
 For these (very small) matrices, our upper bound and the Euclidean bound both perform well when $k$ is large, but for small values of $k,$ our upper bound is much better.


\smallskip
\subsection{Words avoiding $7/3$-powers}

Our last application comes from language Theory.  It has been recently shown that the asymptotic growth of some specific languages can be approximated by joint spectral characteristics of matrices \cite{JPB,J}.
In \cite{BCJ}, this idea has been applied to the so-called language of $7/3$-free words (reads as  ``seven thirds-free words'').  These are words in which any subword is never repeated more than $7/3$ times in some sense.  See~\cite{BCJ} for more information.  The matrices corresponding to that language have size $227\times 227.$  In this reference, the authors analyze the joint and lower spectral radii of that set of matrices, but, in view of the large size of the matrices, nothing is said on the value of the Lyapunov exponent.
  The results are shown on Fig. \ref{fig-7thirds}.  From this figure it appears that the true value lies within the interval $[3.11,4.28].$ We applied our techniques in order to estimate this value.
\begin{figure}
\centering
\includegraphics[scale = 0.5]{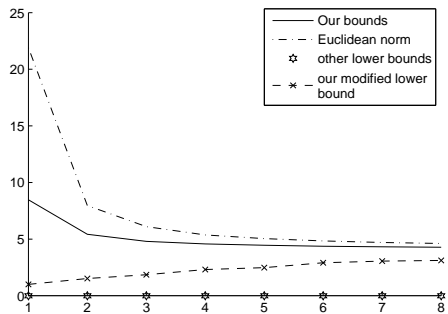}
\caption{\footnotesize{Evolution of the different bounds for the matrices corresponding to the language of $7/3$-free words.}} \label{fig-7thirds}\end{figure}
  The matrices in this application have zero rows and zero columns, therefore we use the modified lower bound $\tilde \beta_k$ defined by~(\ref{betapos-mod}).

\medskip

\subsection{Randomly generated matrices}

\begin{figure}\label{fig-random}
\centering
\begin{tabular}{cc}
\includegraphics[scale = 0.3]{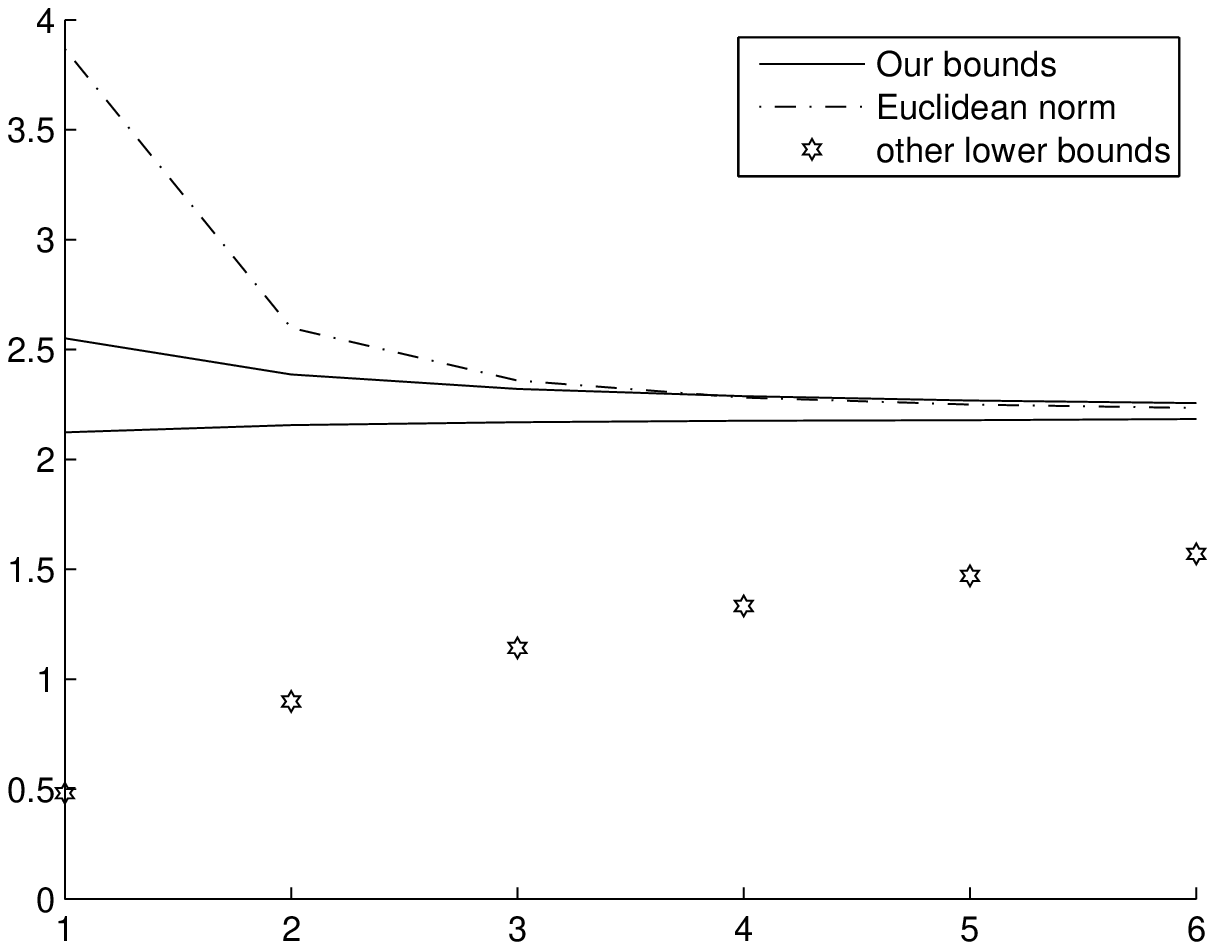}&
\includegraphics[scale = 0.3]{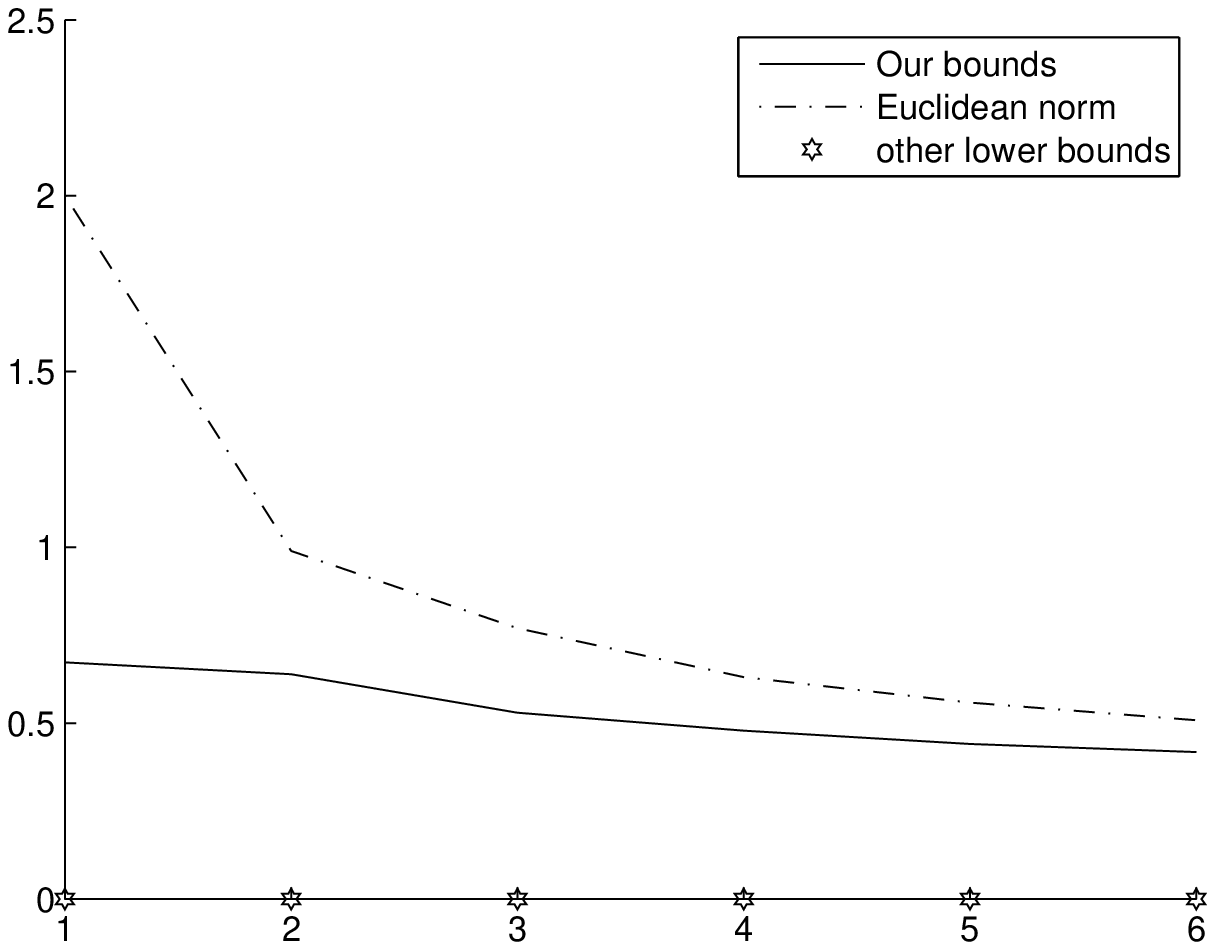} \\(a)&(b)\\ \\
\includegraphics[scale = 0.3]{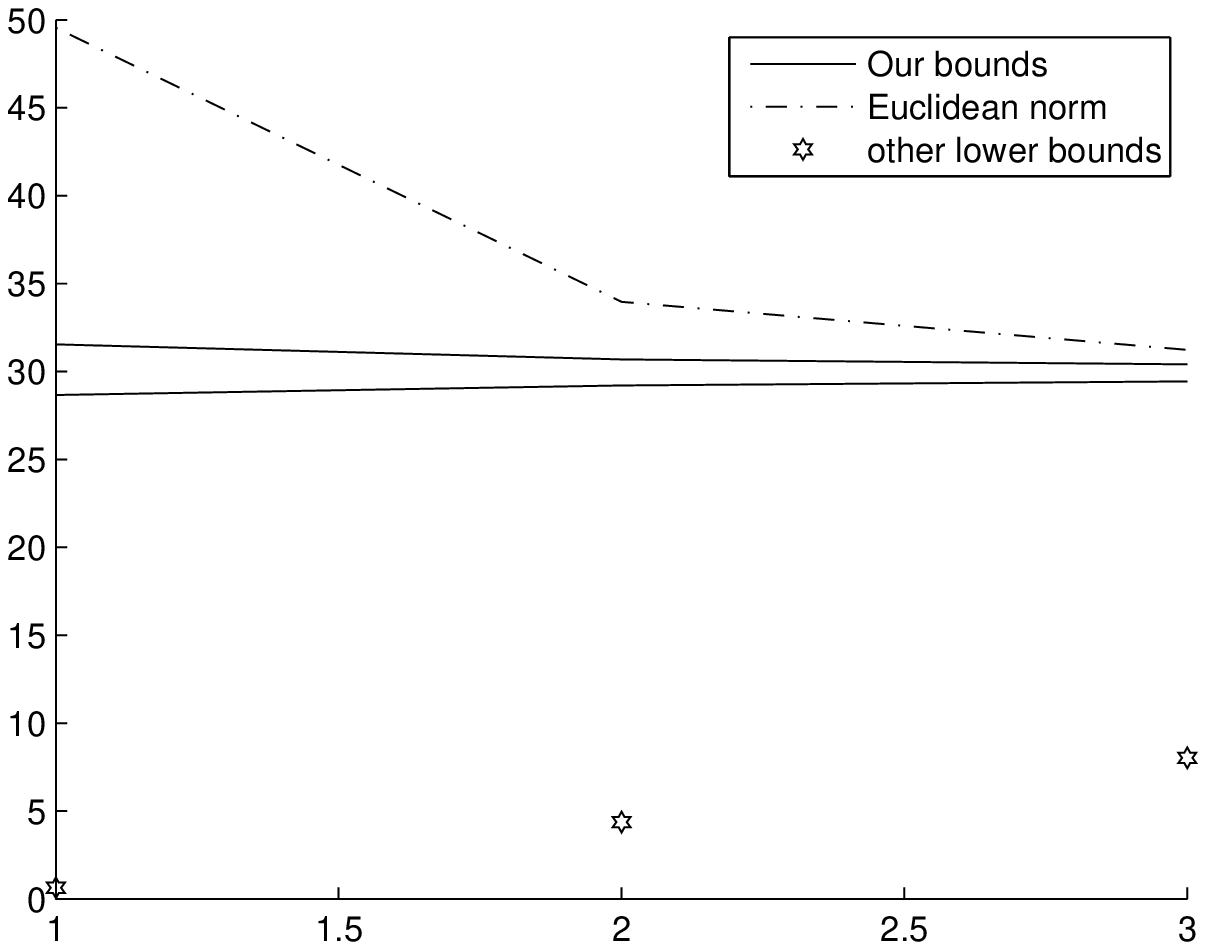}&
\includegraphics[scale = 0.3]{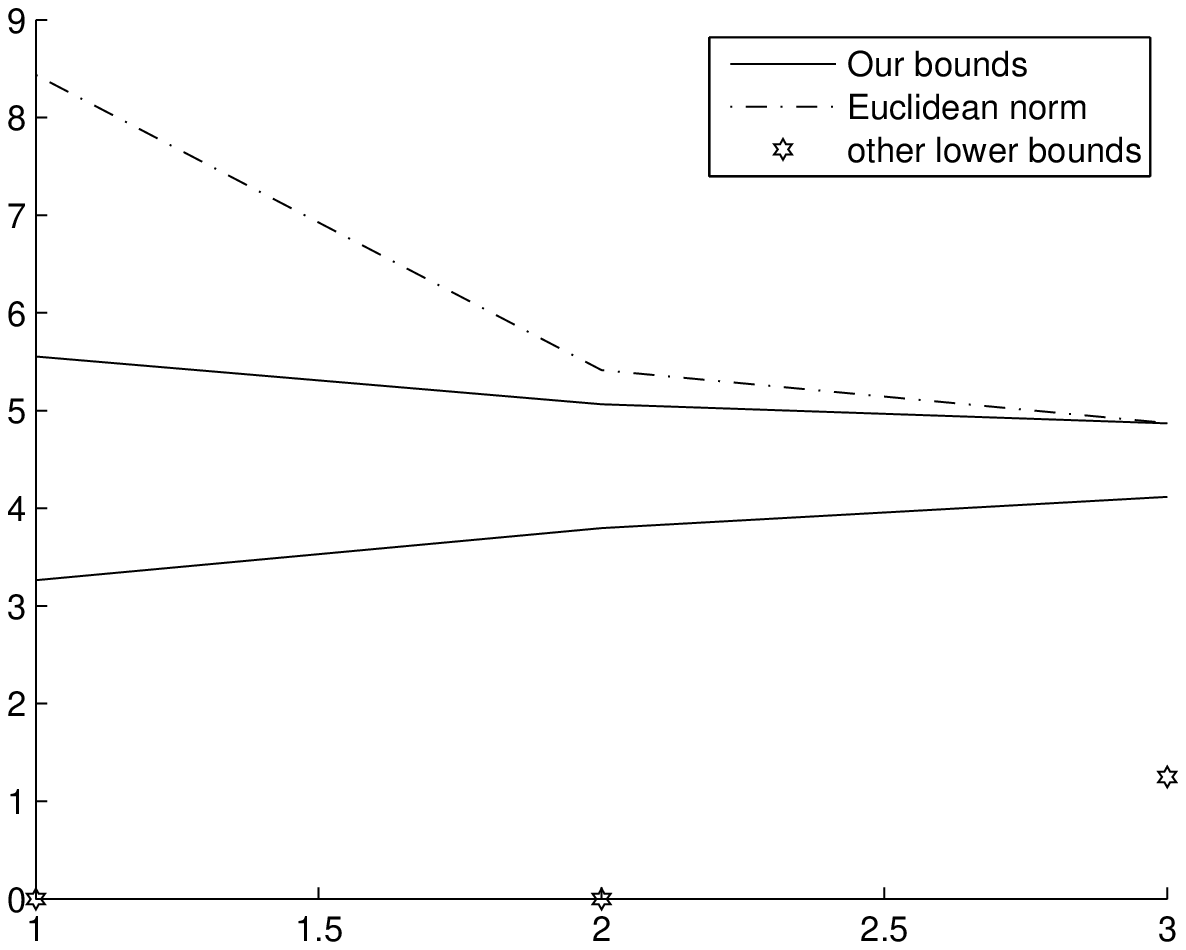}\\(c) & (d)
\end{tabular}
\caption{\footnotesize{Evolution (with the length of the computed products) of the different bounds for random matrices of dimension $5$ ((a) and (b)) and
$60$ ((c) and (d)). In ((b) and (d)), we sparsify the matrices by independently putting each entry to zero with a probability $5/7;$ the other entries are i.i.d homogeneously between zero and one.}}
\end{figure}

\subsubsection{Nonnegative matrices} We ran our algorithms on randomly selected nonnegative matrices.  We report here the results of computation on two different sizes: reasonably small matrices of dimension $5,$ and large matrices of dimension $60.$  For both these sizes, we generated dense and sparse matrices. For the dense matrices, the entries are drawn uniformly and independently between zero and one, while for the second case, we sparsify the matrices by putting each entries to zero with a probability $p=5/7.$  For matrices of size $60,$ computing long products rapidly becomes prohibitive.  As one can see on Fig. \ref{fig-random} (c) and (d), already with products of length~$1$ we have a fairly good approximation of the Lyapunov exponent, while the estimate with the Euclidean norm is still far from convergence.
As one can see in figure (b), if the matrices become so sparse that zero rows can appear, then our lower bound $\beta_k$
may vanish. In this case we apply the modified bound $\tilde \beta_k$. Table \ref{table-random} summarizes the scalability of our method
when the size of the matrices increases. For the same effort of computation, the accuracy seems to be somewhat independent of the dimension.
\begin{table}
\begin{center}
\begin{tabular}{|l|r|c|c|c|}
\hline
size  & $k$ & lower bound& upper bound & accuracy \\
\hline
10 &12 & 5.175 &5.205& 0.6\% \\
20&12 &10.16&10.22 &0.6\% \\
30&11 &15.28&15.35 & 0.4\%\\
40&11 &19.74&19.78  &0.2\% \\
50&11 &24.67 &24.83&0.7\%\\
\hline
\end{tabular}
\caption{\footnotesize{Results of our algorithm on random pairs of matrices. For all matrices, each entry
was drawn i.i.d. at random homogeneously between zero and one. In the second column, $k$ indicates the length of the
products computed in order to derive the bounds in columns $3$ and $4.$ }}
\label{table-random}
\end{center}
\end{table}
\newpage

 We provide the results of computations (the lower bound $\beta_k$ and the upper bound $\alpha_k$) for diverse sizes of matrices, as well as the relative accuracy obtained.  We also mention the maximal length of the products computed in order to reach this accuracy.

\medskip

\subsubsection{Nonnegative matrices with zero columns} In  Fig. \ref{fig-rand15zerolines} we demonstrate applications of the modified lower bound~(\ref{betapos-mod})
for matrices with zero columns, when $\beta_k \equiv -\infty$.

\medskip

\begin{figure}
\centering
\begin{tabular}{cc}
\includegraphics[scale = 0.35]{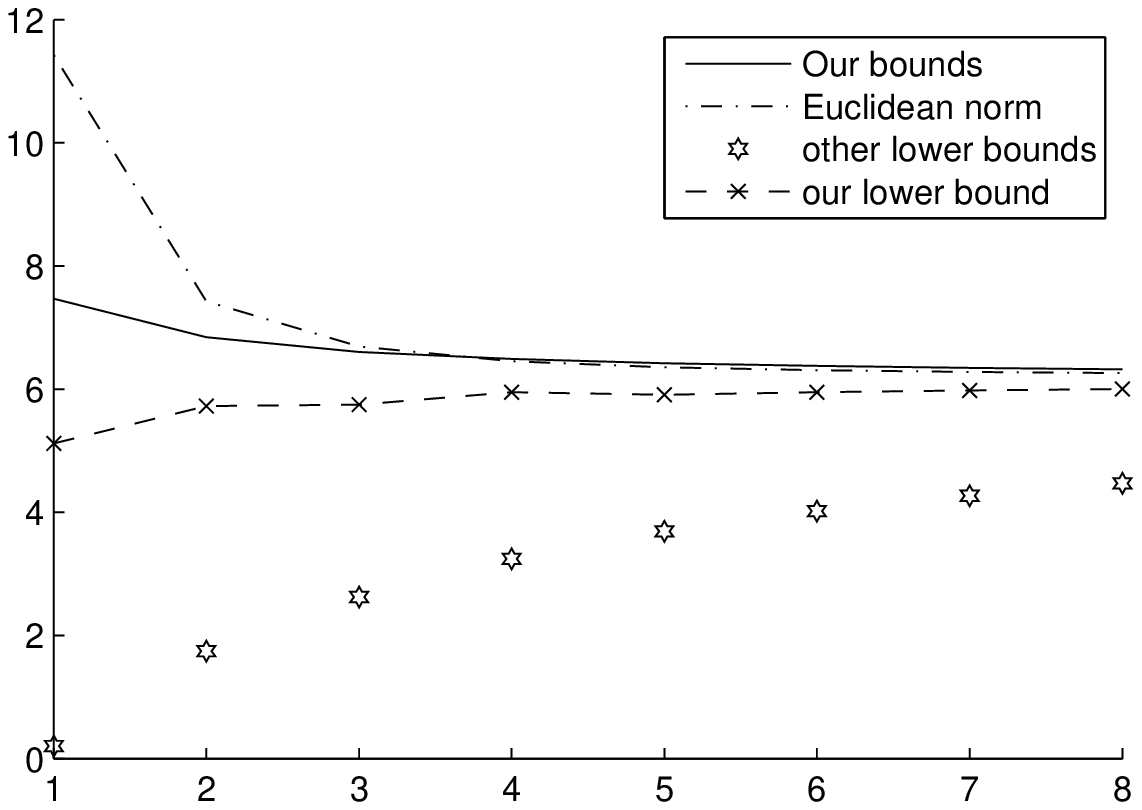}&
\includegraphics[scale = 0.35]{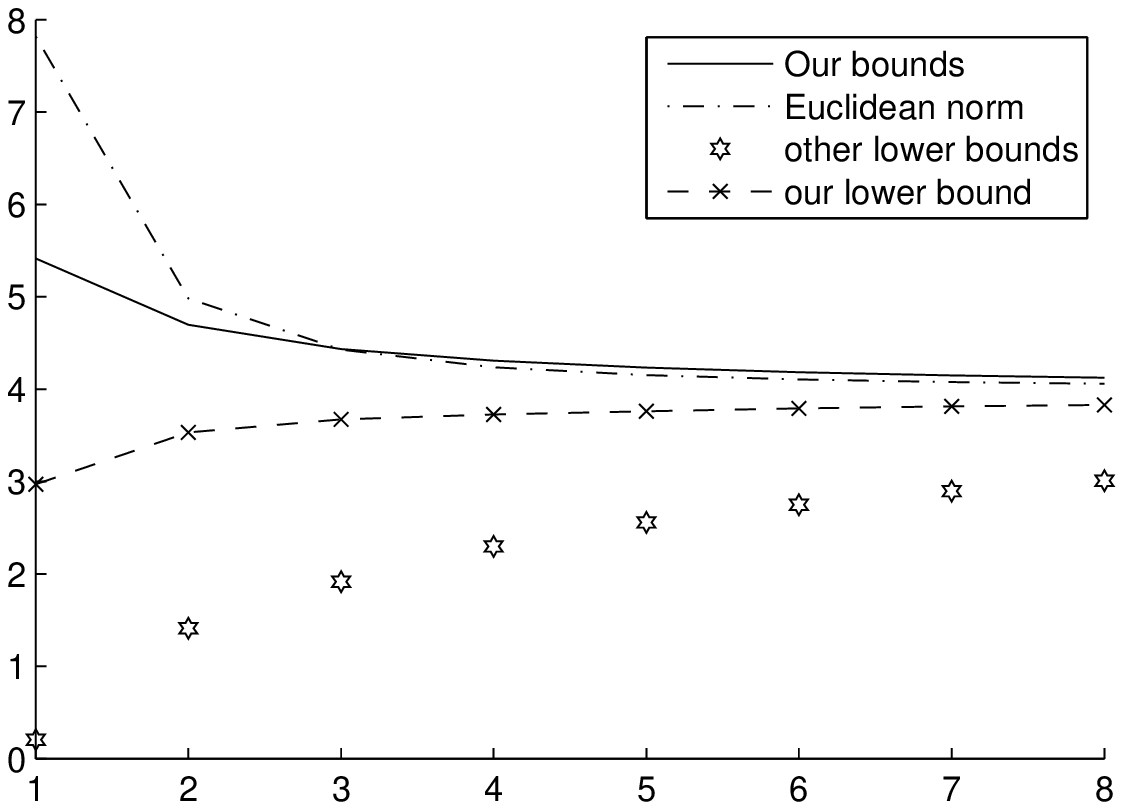} \\(a)&(b)
\end{tabular}
\caption{\footnotesize{Evolution of the different bounds in the situation where some matrices have zero columns. The matrices have dimension $15,$ while $5$ random columns are put to zero.  In the right hand side, the matrices are randomly generated and then sparsified.}} \label{fig-rand15zerolines}
\end{figure}

\subsubsection{Matrices with negative entries} When the matrices do not have only nonnegative entries, no algorithm is known  to squeeze the Lyapunov exponent between a lower and an upper bound.  Of course, the upper iterative estimate~(\ref{euc}) is always available, for any choice of the norm,
  but a good lower estimate converging to $\lambda$ as $k \to \infty$, most likely, does not exist at all (see Section~V).
  The only thing that can be done in this situation is to improve the upper bound. In Proposition~\ref{p20} we showed that
  for every~$k$ the upper bound $\Gamma_k(I)$ is better than~(\ref{euc}) with the Euclidean norm.
Here we compare these bounds for randomly generated matrices, whose entries were all i.i.d. homogeneously between $-0.5$ and $0.5.$  We report in Table \ref{table-nonpositive} the results for $k=1$ for matrices of size $10,20,30,40.$
In Fig. \ref{fig-nonpositive}, we show the accuracy of the two upper bound for a randomly generated $30\times 30$ matrix.  As one can see, $\Gamma_k(I)$ is small already for $k=1,$  while for $k=3$ the Euclidean bound is still far from having converged.

\begin{figure}
\centering
\includegraphics[scale = 0.5]{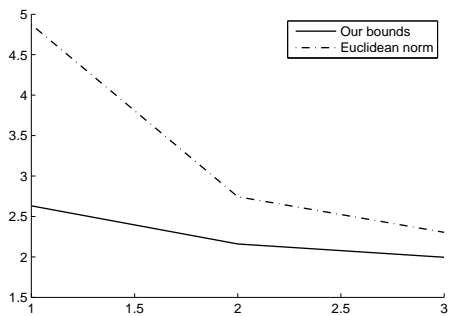}
\caption{\footnotesize{Evolution of our upper bound $\Gamma_k(I)$ versus the classical Euclidean bound for a randomly generated $30\times 30$ matrix with positive and negative entries.}} \label{fig-nonpositive}\end{figure}

\begin{table}
\begin{center}
\begin{tabular}{|c|c|c|}
\hline
size  & our upper bound& Euclidean upper bound \\
\hline
10 &1.5   &2.4 \\
20 &2.2  &4.2 \\
30 &2.7  &4.9 \\
40 &3  &5.7 \\

\hline
\end{tabular}
\caption{Results of our algorithm compared with the Euclidean bound on random pairs of matrices with negative entries. For all matrices, each entry
was drawn i.i.d. at random homogeneously between $-0.5$ and $0.5.$ The algorithms are applied directly on the sets of matrices (i.e. $k=1$).}
\label{table-nonpositive}
\end{center}
\end{table}

\newcommand{\noopsort}[1]{} \newcommand{\singleletter}[1]{#1}

\end{document}